\documentclass[reqno,11pt]{amsart}
\usepackage[utf8]{inputenc}
\DeclareUnicodeCharacter{0301}{}

\calclayout

\usepackage{amsthm,amsfonts,amstext,amssymb,mathrsfs,amsmath,latexsym,mathtools,mathabx}
\usepackage{enumerate}
\usepackage{bm,accents}
\usepackage{mdwlist}
\usepackage[abs]{overpic}
\usepackage[mathscr]{euscript}
\usepackage{mathrsfs}
\DeclareMathAlphabet{\mathpzc}{OT1}{pzc}{m}{it}
\usepackage{parskip}
\usepackage{hyperref}
\usepackage{comment}

\usepackage{tikz}
\usetikzlibrary{shapes.geometric, arrows}
\usetikzlibrary{calc}

\usepackage[font=small]{caption}
\usepackage[labelformat=simple]{subcaption}

\usepackage{color}
\usepackage{esint} 
\usepackage[colorinlistoftodos,prependcaption,textsize=small]{todonotes}

\definecolor{red}{RGB}{255,0,0}
\definecolor{green}{RGB}{0,100,0}
\definecolor{blue}{RGB}{0,0,255}

\usepackage{cleveref}

\crefname{equation}{equation}{equations}
\crefname{figure}{Figure}{Figures}

\theoremstyle{plain}
\newtheorem{thm}{Theorem}[section]

\newtheorem{prop}[thm]{Proposition}
\newtheorem{lemma}[thm]{Lemma}
\newtheorem{cor}[thm]{Corollary} 

\theoremstyle{definition}
\newtheorem{definition}[thm]{Definition}
\newtheorem{example}[thm]{Example}

\theoremstyle{remark}
\newtheorem{remark}[thm]{Remark}
\numberwithin{equation}{section}

\numberwithin{equation}{section}

\renewcommand{\Re}{\mathop{\rm Re}}
\renewcommand{\Im}{\mathop{\rm Im}}

\DeclarePairedDelimiter\floor{\lfloor}{\rfloor}

\newcommand{\N}{\mathbb{N}}
\newcommand{\C}{\mathbb{C}}

\newcommand{\R}{\mathbb{R}}
\newcommand{\Z}{\mathbb{Z}}

\newcommand{\Boh}{\mathcal{O}}

\definecolor{apricot}{rgb}{0.98, 0.81, 0.69}

\definecolor{greenpie}{rgb}{0.69, 0.95, 0.76}

\newcommand*{\defeq}{\mathrel{\vcenter{\baselineskip0.5ex \lineskiplimit0pt
                     \hbox{\scriptsize.}\hbox{\scriptsize.}}}
                     =}
\newcommand*{\eqdef}{=\mathrel{\vcenter{\baselineskip0.5ex \lineskiplimit0pt
                     \hbox{\scriptsize.}\hbox{\scriptsize.}}}
                     }   

\newcommand{\TT}{\mathcal{T}}
\newcommand{\dd}{\mathrm{d}}

\newcommand\restr[2]{{
		\left.\kern-\nulldelimiterspace 
		#1 
		\vphantom{\big|} 
		\right|_{#2} 
}}

\title[Elliptic orthogonal polynomials and OPRL]{Elliptic orthogonal polynomials and OPRL }

\author[V.~Alves]{Victor Alves}

\address[VA]{Instituto de Ciências Matemáticas e de Computação (ICMC), Universidade de São Paulo (USP), São Paulo, Brazil }

\email{victorjulio@usp.br}

\author[A. Mart\'{\i}nez-Finkelshtein]{Andrei Mart\'{\i}nez-Finkelshtein}

\address[AMF]{Department of Mathematics, Baylor University, Waco, TX 76706, USA, and Department of Mathematics, University of Almer\'{\i}a, Almer\'{\i}a, Spain}

\email{A\_Martinez-Finkelshtein@baylor.edu}

\date{\today}



\begin{document}

\begin{abstract}
    
    We explore a class of meromorphic functions on elliptic curves, termed \emph{elliptic orthogonal a-polynomials} ($a$-EOPs), which extend the classical notion of orthogonal polynomials to compact Riemann surfaces of genus one. Building on Bertola's construction of orthogonal sections, we study these functions via non-Hermitian orthogonality on the torus, establish their recurrence properties, and derive an analogue of the Christoffel--Darboux formula. We demonstrate that, under real-valued orthogonality conditions, $a$-EOPs exhibit interlacing and simplicity of zeros similar to orthogonal polynomials on the real line (OPRL). Furthermore, we construct a general correspondence between families of OPRL and elliptic orthogonal functions, including a decomposition into multiple orthogonality relations, and identify new interlacing phenomena induced by rational deformations of the orthogonality weight. 
\end{abstract}

\maketitle

\section{Introduction} \label{sec:intro}

Orthogonal polynomials on the real line (OPRL) play a central role in analysis, mathematical physics, and approximation theory. In recent years, there has been growing interest in generalizing this concept to higher-dimensional and more complex settings, particularly to functions defined on compact Riemann surfaces. Among these, \emph{elliptic orthogonal functions}, defined on tori (genus one Riemann surfaces), provide a rich and tractable class of objects that bridge classical analysis with algebraic geometry.

A foundational approach to such generalizations was developed by Bertola~\cite{Bertola2021a}, who introduced \emph{orthogonal sections}: meromorphic functions on compact Riemann surfaces satisfying orthogonality relations defined via generalized Cauchy kernels and Weyl--Stieltjes transforms. These constructions, although technically involved, are intimately connected with Padé approximation and enable the formulation of Riemann--Hilbert problems suitable for asymptotic analysis~\cite{Bertola2022a}.

In this work, we focus on the simplest nontrivial case: the genus $g=1$ setting of elliptic curves, or equivalently, tori. We define a natural analogue of orthogonal polynomials on these surfaces, which we call \emph{elliptic orthogonal $a$-polynomials} (abbreviated as $a$-EOPs). These are meromorphic functions on the torus with prescribed poles at two points: the ``point at infinity'' and a fixed ``anchor point'' $a$ (which is the parameter $a$ in the acronym $a$-EOPs). They satisfy orthogonality conditions with respect to a complex weight on a closed contour $\Gamma$ on the torus, and form a basis in appropriate function spaces.

   In 1937, J.~Dunham \cite{Dunham37}
   studied orthogonal polynomials on a planar algebraic curve.
   Much more recent is the work of Spicer, Nijhoff, and Kamp  \cite{SpicerNijhoff11}, where they considered the orthogonalization of the sequence
    \[
        1, x, y, x^2, xy, x^3, x^3y, x^4, x^4y, x^5, x^5y, \ldots  \ldots 
    \]
    subject to the constraint that $x$ and $y$ are related by the elliptic equation
    \[
        y^2 = 4 x^3 - g_2 x - g_3;
    \]
    they called the corresponding orthogonal sequence the \textit{two-variable elliptic orthogonal polynomials}.
    
    In another direction, Fasondini, Olver and Xu also developed many bivariate orthogonal polynomials in planar algebraic curves of form $y^m = p(x)$ where $m = 1,2$ and $p$ is a polynomial. and quadratic surfaces of revolution \cite{FasondiniOlverXi2023a, FasondiniOlverXu23b, OlverXu2019, OlverXu2020, OlverXu2021}.

    The most recent and similar construction to ours appeared in Desiraju, Latimer and Roffelsen \cite{Desiraju2024}, where they consider elliptic orthogonal polynomials without the ``anchor point'' $a$ (which forces them to skip the ``polynomials'' of degree $1$). See also \cite{DesirajuLahiry2025} for an investigation on their corresponding recurrence relations and Christoffel-Darboux formula.

Our main contributions are threefold:
\begin{itemize}
    \item We provide a simplified construction of $a$-EOPs, establish their structural properties, and derive recurrence relations and a Christoffel--Darboux formula tailored to the elliptic setting.
    \item We investigate the case of real-valued orthogonality, proving that under natural symmetry and positivity assumptions, the zeros of $a$-EOPs are simple and interlace, thus mirroring classical results from the theory of OPRL.
    \item We develop a lifting mechanism from OPRL to the elliptic setting, and a decomposition framework that expresses $a$-EOPs in terms of multiple orthogonality conditions on the real line. This leads to novel interlacing results for rationally modified OPRL families, including examples involving Jacobi polynomials.
\end{itemize}

The paper is organized as follows. In Section~\ref{sec:torus}, we recall the basic theory of elliptic functions and define the relevant function spaces. Section~\ref{sec:nonH} introduces non-Hermitian orthogonality on the torus and derives foundational results. In Section~\ref{sec:realvalued}, we analyze the real-valued setting, proving zero interlacing theorems. Section~\ref{sec:OPRL-EOP} is devoted to establishing connections between $a$-EOPs and classical OPRL, with explicit examples and a general decomposition theorem.

\section{Polynomials on a torus}\label{sec:torus}

Elliptic functions are doubly-periodic meromorphic functions on $\C$. If we denote\footnote{\, We are following the notation of \cite{NIST:DLMF}.} their periods by $2\omega_1$ and $2\omega_3$, with the standard assumption that they are $\R$-linearly independent, then they can be regarded as meromorphic functions on a torus $\mathcal T$ defined as the quotient space of $\C$ by the integer lattice of periods:
		\begin{equation}\label{TorusRep1}
			\TT  :=  \C / \Lambda, \quad \Lambda  \defeq   2\omega_1\Z +   2\omega_2\Z. 
		  \end{equation}
It is known that $\TT$ can be endowed with a complex structure, becoming a compact Riemann surface of genus 1. We can establish a one-to-one correspondence between $\TT$ and the \textit{fundamental parallelogram} 
\begin{equation}
    \label{parallelogramm}
    \Omega:= 2\omega_1 [0,1) +  2 \omega_3[0,1)
\end{equation}
by identifying the opposite edges. The set $\mathcal L (n\cdot 0 )$ of elliptic functions holomorphic on $\Omega\setminus \{0\}$ with a pole at $0$ of order $\le n\in \Z_{\ge 0}:= \mathbb N \cup \{0\}$ is a vector space. By a well-known property of elliptic functions, the set $\mathcal L (1\cdot 0 ) $ contains only constant functions. 

Since a polynomial of degree $\le n\in \Z_{\ge 0}$ is a meromorphic function on the Riemann sphere $\overline \C:= \C \cup \{\infty\}$ with its only pole of degree at most $n$, at $\infty$, the set $\mathcal L (n\cdot 0 )$ is a natural analogue on $\TT$ of polynomials of degree $\le n$. 

A well-know representative of elliptic functions is the \textit{Weierstrass $\wp$ function}
        \[
            \wp(z) = \frac{1}{z^2} + \sum_{\lambda \in \Lambda \setminus \{0\}} \left( \frac{1}{(z-\lambda)^2} - \frac{1}{\lambda^2}\right),
        \]
 which is even, and its only poles are double poles at $\Lambda$ (so that $\wp\in \mathcal L (2\cdot 0 )$), with
        \begin{equation} \label{asymptWP}
            \wp(z) = \frac{1}{z^2} + \Boh(z^2), \quad z \to 0.
        \end{equation}
The function $\wp$ satisfies a first-order differential equation \cite[\S 23.3]{NIST:DLMF} 
        \begin{equation}\label{eq:wpODE}
            (\wp'(z))^2 = 4(\wp(z) - e_1)(\wp(z) - e_2)(\wp(z)-e_3)=4(\wp(z))^3-g_2 \, \wp(z)-g_3 
        \end{equation}
        with constants
        \begin{equation}
            \label{defE}
                e_1=\wp\left(\omega_1 \right), \quad e_2=\wp\left(\omega_1 + \omega_3 \right), \quad e_3=\wp\left(\omega_3 \right),
        \end{equation}
        and the \textit{Weierstrass invariants}
       \begin{equation}
            \label{defG}
        	g_2  =60 \sum_{\lambda \in \Lambda \setminus \{0\}} \frac{1}{ \lambda^4},\qquad  g_3 =140 \sum_{\lambda \in \Lambda \setminus  \{0\}} \frac{1}{\lambda^6}.
   \end{equation}
Function $\wp$, and hence, the lattice $\Lambda$, do not uniquely determine the half-periods $\omega_1$, $\omega_2$. However, by \eqref{eq:wpODE} and  \eqref{defG}, for each $\wp$, the values $e_1, e_2, e_3$ and the Weierstrass invariants are unique, and we can write $\wp (\cdot; \Lambda)$ or $\wp(\cdot; g_2, g_3)$ when we need to specify the parameters. 

Periods $\omega_1$, $\omega_3$ and constants $e_j$ in \eqref{defE} are related by
\begin{equation} \label{formulaperiods}
\begin{split}
2 \omega_1=\int_{e_1}^{\infty} \frac{d u}{\sqrt{\left(u-e_1\right)\left(u-e_2\right)\left(u-e_3\right)}}=\int_{e_3}^{e_2} \frac{d u}{\sqrt{\left(e_1-u\right)\left(e_2-u\right)\left(u-e_3\right)}}, \\
2 \omega_3=i \int_{e_2}^{e_1} \frac{d u}{\sqrt{\left(e_1-u\right)\left(u-e_2\right)\left(u-e_3\right)}}=i  \int_{-\infty}^{e_3} \frac{d u}{\sqrt{\left(e_1-u\right)\left(e_2-u\right)\left(e_3-u\right)}} ,
\end{split}
\end{equation}
see \cite{NIST:DLMF}.

Equation \eqref{eq:wpODE} allows the identification of $\TT$ with an elliptic (cubic) plane curve 
$\mathcal C$, the compact Riemann surface of
        \[
            y^2 = 4x^3 - g_2\, x - g_3.
        \]
Indeed,   
\begin{equation}\label{eq:WeierstrassParametrization}
    z \quad \mapsto \quad (x,y) = \left(\wp(z), - \frac{\wp'(z)}{2}\right), \quad \wp(z)=\wp(z; g_2, g_3),
\end{equation}
is a biholomorphic map (\textit{Weierstrass' parametrization}) from $\TT$ onto $\mathcal C$. 

Addressing the lack of non-constant functions in $\mathcal L (1\cdot 0 )$, we could use a formal primitive of $\wp$. Recall that the \textit{Weierstrass zeta function} is defined as
        \[
            \zeta(z) \defeq \frac{1}{z} - \int_0^z \left[ \wp(u) - \frac{1}{u^2}\right] \, \dd u=\frac{1}{z}+\sum_{\lambda \in \Lambda \setminus \{0\}}\left[\frac{1}{z-\lambda}+\frac{1}{\lambda}+\frac{z}{\lambda^2}\right] , \quad z \in \C,
        \]
        where the path of integration in $\C$ does not intersect $\Lambda \setminus \{0\}$. It is easy to verify that $\zeta$ is an odd function, with a simple pole at $z=0$, 
        	\[
		\zeta(z)=	 \frac{1}{z} + \Boh \left(z\right), \quad z \to 0,  
		\]
that satisfies that $\zeta' = - \wp$. However, it is no longer doubly periodic (and, thus, not an elliptic function). The following identities are well-known (e.g., \cite[$\mathsection 23.2$]{NIST:DLMF}):
\[
            \zeta(z +   2\omega_i) = \zeta(z) +   2 \eta_i, \quad z \in \C, \quad \text{where} \quad \eta_i = \zeta(\omega_i), \quad i=1,3.
        \]
    Since the additive increments do not depend on $z$, a difference of two zeta functions is elliptic, so that for any $a\in \Omega\setminus \{0\}$,
        \[
            \zeta(z) - \zeta(z-a) + \mathrm{const}
        \]
is an elliptic function with two simple poles in $\Omega$: one at $z=0$ and another one, at $z=a$. 

The previous discussion motivates the following notion:
   \begin{definition} \label{def:21}
 Let  $a \in \Omega \setminus \{0\}$.  For $n\in \Z_{\ge 0}$, $\mathcal L (n\cdot 0 +a)$ denotes the vector space  of elliptic functions with periods \eqref{TorusRep1}, holomorphic in $\Omega\setminus \{0,a\}$, with a possible pole at $0$ of order $\le n$, and at most a simple pole at $a$. Elements of $\mathcal L (n\cdot 0 +a)$ are \emph{elliptic $a$-polynomials} on $\TT$. 
            
Moreover, if for $f\in \mathcal L (n\cdot 0 +a)$, 
		\[
		f(z)=	 \frac{c}{z^n} + \Boh \left(\frac{1}{z^{n-1}}\right), \quad z \to 0, \quad c\in \C\setminus \{0\},
		\]
we say that $n$ is the \emph{polynomial degree} of $f$, and $c$ is its \emph{leading coefficient}. When $c = 1$, $f$ is said to be
  \textit{monic}. 
        \end{definition}
        \begin{remark}
By Definition~\ref{def:21}, the number of poles of $f$ on $\TT$, counted with multiplicity, can exceed at most in 1 its polynomial degree; these values actually match if and only if $f$ is holomorphic at $a$.

Notice also that the limit case $a\to 0$ implies excluding the functions with polynomial degree degree $1$ from our construction, building a theory that resembles that of the exceptional orthogonal polynomials on $\C$. This construction was considered in \cite{DesirajuLahiry2025, Desiraju2024}, where the authors study elliptic orthogonal functions with poles at $0$ only.
        \end{remark}

Clearly, 
\begin{equation} \label{defB1}
b_0(z) \defeq 1, \quad     b_1(z;a) \defeq \zeta(z) - \zeta(z-a) - \zeta(a) = -\frac{1}{2} \frac{\wp'(z) + \wp'(a)}{\wp(z) - \wp(a)},
\end{equation}
    form a monic basis of $\mathcal L (0 +a)$, see \cite{Bertola2021a}; function $b_1$ is an analogue of the Cauchy kernel on $\TT$, see \cite{Desiraju2024}, and the alternative expression for $b_1$ follows from the addition theorem for the $\zeta$ function, \cite[formula (23.10.2)]{NIST:DLMF}. 
We can complement these functions with 
\begin{equation}
    \label{basis1}
    b_{2k}(z) \defeq \wp^k(z), \quad b_{2k+1}(z) \defeq - \frac{1}{2} \wp'(z) \wp^{k-1}(z), \quad k\in \mathbb N,
\end{equation}
so that $b_j \in \mathcal L (j\cdot 0 )$, $ j\ge 2$. Obviously, $\{b_j\}_{j\ge 0}$ forms a basis of $\mathcal L (\infty\cdot 0 +a ):=\bigcup_{j\ge 0} \mathcal L (j\cdot 0 +a )$.

\begin{remark}
Another option is to use the derivatives of the $\wp$ function, and complete $b_0$ and $b_1$ in \eqref{defB1} with  
$
  c_k \wp^{(k-2)}(z)
$, 
where $c_k$ are some normalizing constants. Identities \cite[\S 23.3]{NIST:DLMF}
$$
\wp^{\prime \prime}(z)=6 \wp^2(z)-\frac{1}{2} g_2 , \quad  \wp^{\prime \prime \prime}(z)=12 \wp(z) \wp^{\prime}(z)
$$
show that these two choices are equivalent, and that we can restrict our attention to the set $\{b_j\}_{j\ge 0}$. 
\end{remark}

\section{Non-Hermitian orthogonality on a torus}\label{sec:nonH}

We want to address the notion of orthogonality on $\TT$ with respect to a weight. The ingredients are a closed contour $\Gamma  $ on the torus $\TT$ and a weight function $W$ defined on $\Gamma$. The only assumptions on $\Gamma$ and $W$ so far are that all the integrals below are well-defined and finite; in what follows, we also impose that $a \in \Omega \setminus \{0\}$, $a\notin \Gamma$.

Given $n\in \Z_{\ge 0}$, an elliptic $a$-polynomial $F_n\in \mathcal L (n\cdot 0 +a)$ is \textit{orthogonal} with respect to $W$ (shortly, $a$-EOP)  if
    \begin{equation}
        \label{deforth}
             \int_\Gamma F_n(s) Q(s) W(s) \, \dd s = 0
    \end{equation}
for any $Q \in \mathcal L ((n-1)\cdot 0 +a)$. 
Notice that this is a non-hermitian orthogonality, which implies that the existence of $F_n\not \equiv 0$ is not guaranteed a prior, and that $\Gamma$ in the integral is freely deformable on $\Omega\setminus\{0, a\}$. From the general theory, 
 \begin{equation}\label{eq:DetFormula}
            F_n(z) = 
            \det 
            \begin{pmatrix}
                \mu_{0,0}   &   \mu_{0,1}   &   \cdots& \mu_{0,n-1} &   \mu_{0,n}   \\
                \mu_{1,0}   &   \mu_{1,1}   &   \cdots& \mu_{1,n-1} &   \mu_{1,n}   \\
                \vdots  &   \vdots  &   \ddots  &   \vdots  &   \vdots  \\
                \mu_{n-1,0}   &   \mu_{n-1,1}   &   \cdots& \mu_{n-1,n-1} &   \mu_{n-1,n}   \\
                b_0(z)  &   b_1(z)  &   \cdots  &   b_{n-1}(z)  &   b_n(z)
            \end{pmatrix},
        \end{equation}
where the bi-moments $\mu_{i,j}$ are given by integrals of products of elements of the basis \eqref{defB1}--\eqref{basis1},
        \[
            \mu_{i,j} \defeq \int_\Gamma b_i(s) b_j(s) W(s) \, \dd s.
        \]        
Moreover, the existence of $F_n$ of the polynomial degree exactly $n$ is guaranteed if all minors 
$$
D_k = D_k(a) \defeq \det \left(\mu_{i,j}\right)_{i,j=0}^{k-1}\neq 0, \quad k\in \N. 
$$
An equivalent expression for $D_k$ is Andréief’s integration formula \cite[Eq. (2.1.10)]{Szego},
\begin{equation}
    \label{andreief}
     D_k = \frac{1}{k!} \idotsint_{\Gamma^k} \left( \det \left[ b_{j-1}(x_i)\right]_{i,j=1}^{k} \right)^2 \,  \prod_{i=1}^k W(x_i) \, d x_1\, \cdots \, d x_k.
\end{equation}
In the case when all $D_k\neq 0$, we can introduce also the ``orthonormal'' $a$-polynomial 
            \begin{equation}
                \label{orthbasis}
             f_n(z):= \frac{1}{\sqrt{D_{n-1} D_n}} \, F_n(z)\in \mathcal L (n\cdot 0 +a),
\end{equation}
            where we choose any branch of the square root. 

An analogue of the well-known three-term recurrence relation satisfied for non-Hermitian orthonormal polynomials on the complex plane is the following result: 
\begin{lemma}\label{lem:5TRR}
    Under the assumption that $D_n\neq 0$ for all $n\in \Z_{\ge 0}$, for each $k \geq 0$, there exist constants $A_k, B_k, C_k \in \C$ such that
            \begin{equation}\label{eq:FiveTermRecurrence}
                     \wp(z)   f_k(z) = A_k  f_{k+2}(z) + B_k  f_{k+1}(z) + C_k  f_k(z) + B_{k-1} 
                      f_{k-1}(z) + A_{k-2}   f_{k-2}(z),
                \end{equation}
                where $f_{-1} \equiv f_{-2} \equiv 0$. Moreover,
                 \begin{equation} \label{formualcoefficients}
                     \begin{split}
                    A_k &= \int_\Gamma \wp(z)  f_k(z)  f_{k+2}(z) W(z)\, d z, \\ 
                    B_k &= \int_\Gamma \wp(z)  f_k(z)  f_{k+1}(z) W(z)\, d z, \\
                    C_k &= 
                    \int_\Gamma \wp(z)  f_k(z)^2 W(z)\, d z.
                    \end{split}
                \end{equation}
            \end{lemma}
            \begin{proof}
                Since $\wp   f_k \in \mathcal L ((k+2)\cdot 0 +a) $ and \eqref{defB1}--\eqref{basis1} is a basis of $ \mathcal L (\infty \cdot 0 +a)$, we can write
                \[
                    \wp(z)   f_k(z) = \sum_{j=0}^{k+2} \lambda_{j,k}  f_j(z) , \quad 
                    \lambda_{j,k} = \int_\Gamma \wp(z)  f_k(z)  f_j(z) W(z)\, d z, \quad j=0,1, \dots, k+2.
                \]
                For $j < k-2$, $\wp   f_j\in \mathcal L ((k-1)\cdot 0 +a)$, and by orthogonality of $  f_j$ to $  \mathcal L ((k-1)\cdot 0 +a)$, $\lambda_{j,k} = 0$. Clearly, $\lambda_{j,k} = \lambda_{k,j}$ for $k-2 \leq j \leq k+2$ and  \eqref{eq:FiveTermRecurrence} follows.
            \end{proof}

            The function 
            \[
                K_n(z,u) := \sum_{j=0}^{n-1}  f_j(z)  f_j(u)
            \]
            is the reproducing kernel for $  \mathcal L ((n-1)\cdot 0 +a)$. Indeed, for $Q  = \sum_{i=0}^{n-1} \lambda_i  f_i \in \mathcal L ((n-1)\cdot 0 +a)$, 
            \[
                \int_\Gamma Q(u) K_n(z,u) W(u)du = \sum_{i,j=0}^{n-1} \lambda_i  f_j(z) \int_\Gamma  f_i(u)   f_j(u) W(u)d u 
                = \sum_{i=0}^{n-1} \lambda_i   f_i(z) = Q(z). 
            \]
As in the classical case, the existence of a finite recurrence relation implies the existence of a simplified expression for the kernel (the \textit{Christoffel-Darboux formula}\footnote{\, See the recent work \cite{DesirajuLahiry2025} containing a similar result.}):
             \begin{lemma}
If $z^2\neq u^2$, the reproducing kernel satisfies the identity
        \begin{equation}\label{eq:CDFormula}
        \begin{split}
            K_n(z,u) &= A_{n-1} \frac{ f_{n+1}(z)f_{n-1}(u) - f_{n+1}(u) f_{n-1}(z) }{\wp(z) - \wp(u)} \\
            &\phantom{=} + A_{n-2} \frac{ f_{n}(z) f_{n-2}(u) - f_{n}(u) f_{n-2}(z)}{\wp(z) - \wp(u)}    + B_{n-1} \frac{ f_{n}(z) f_{n-1}(u) - f_{n}(u) f_{n-1}(z)}{\wp(z) - \wp(u)},
        \end{split}
        \end{equation}
        with $A_k$'s and $B_k$'s given in \eqref{formualcoefficients}.
        
If $z \in \Omega\setminus\{ \omega_1, \omega_2, \omega_3\}$, 
        \begin{equation}
        \label{confluent}
        \begin{split}
            K_n(z,\pm z) &= A_{n-1} \frac{ f_{n+1}'(z)f_{n-1}(z) - f_{n+1}(z) f_{n-1}'(z) }{\wp'(z) } \\
            &\phantom{=} + A_{n-2} \frac{ f_{n}'(z) f_{n-2}(z) - f_{n}(z) f_{n-2}'(z)}{\wp'(z) }   + B_{n-1} \frac{ f_{n}'(z) f_{n-1}(z) - f_{n}(z) f_{n-1}'(z)}{\wp'(z)}.
        \end{split}
        \end{equation}
  Finally, if $z \in \{ \omega_1, \omega_2, \omega_3\}$,
        \begin{equation}
        \label{confluent2}
        \begin{split}
            K_n(z,\pm z) &= A_{n-1} \frac{ f_{n+1}''(z)f_{n-1}(z) - f_{n+1}(z) f_{n-1}''(z) }{\wp''(z) } \\
            &\phantom{=} + A_{n-2} \frac{ f_{n}''(z) f_{n-2}(z) - f_{n}(z) f_{n-2}''(z)}{\wp''(z) }   + B_{n-1} \frac{ f_{n}''(z) f_{n-1}(z) - f_{n}(z) f_{n-1}''(z)}{\wp''(z)}.
        \end{split}
        \end{equation}
    \end{lemma}

    \begin{proof}
        Let $G_{j,k}(z,u) \defeq f_{k+j}(z) f_k(u) - f_{k+j}(u) f_k(z)$. We multiply equation \eqref{eq:FiveTermRecurrence} by $f_k(u)$ and subtract the corresponding identity with the roles of $z$ and $u$ exchanged to obtain
        \[
            [ \wp(z) - \wp(u)] f_k(z) f_k(u) 
            = A_k G_{2,k}(z,u) + B_k G_{1,k}(z,u) - B_{k-1} G_{1,k-1}(z,u) - A_{k-2} G_{2,k-2}(z,u).
        \]
      A summation of the identities above from $k=0$ to $n-1$ gives us
        \[
            K_n(z,u) = \frac{B_{n-1} G_{1,n-1}(z,u) + A_{n-1} G_{2,n-1}(z,u) + A_{n-2} F_{2,n-2}(z,u)}{\wp(z) - \wp(u)},
        \]
       equivalent to \eqref{eq:CDFormula}. Equations \eqref{confluent} and \eqref{confluent2} are obtained from \eqref{eq:CDFormula} taking limits as $u\to \pm z$.
    \end{proof}

\section{Real-valued orthogonality on a torus}\label{sec:realvalued}

 The non-Hermitian orthogonality on the complex plane reduces to the Hermitian one when the path of integration belongs to $\R$ and the weight function is real-valued and positive. In this case, the behavior of zeros of the orthogonal polynomials is well known. 

 In analogy with these considerations, we can assume that $\Gamma \subset \TT$ is a contour on which both the elements of the basis $\{b_j\}$ and the weight $W$ are real-valued, and $W>0$.  
 Andréief’s integration formula \eqref{andreief} shows that $D_n > 0$ for all $n\in \Z_{\ge 0}$, which, as we have seen, implies the existence of the orthonormal basis $\{f_n\}$ \eqref{orthbasis}, $f_n\in \mathcal L (n\cdot 0 +a)$. 

By \eqref{defB1}--\eqref{basis1}, for the real-valuedness of the basis $\{b_j\}$ it is sufficient to assume that $\wp(a)\in \R$ and that $\wp$ and $\wp'$ are real-valued on $\Gamma$. By \eqref{eq:DetFormula}, in this case $f_j$ are real-valued on $\Gamma$. 

A natural choice of $\Gamma$ is the analogue of the real line on $\TT$, namely, the interval $[0, 2\omega_1)\in \Omega$. Since the Weierstrass functions take real values on the real axis if and only if the invariants $g_2, g_3 \in \R$, it is natural to impose this condition. As a consequence of \eqref{defE}, $e_1\in \R$. For simplicity, we consider the case when all three values $e_i$ are real and distinct. This is the case of the discriminant
$$
\Delta:=g_2^3-27 g_3^2=16\left(e_2-e_3\right)^2\left(e_3-e_1\right)^2\left(e_1-e_2\right)^2>0,
$$
and 
\begin{equation}
    \label{positionE}
   e_3<e_2<e_1,   \quad  e_3<0<e_1,
\end{equation}
see \cite[\S 23.9]{NIST:DLMF}. By \eqref{formulaperiods}, $\omega_1, \tau:=\omega_3/i>0$, and $\Omega$ is a rectangle. Properties of the the Weierstrass functions show that both $\wp$ and $\wp'$ are real-valued on $\Gamma$ when $\Gamma=\gamma_i$, $i=1,2$, with 
\begin{equation}
    \label{curvesDef}
    \gamma_1 :=  [0,2 \omega_1), \quad \gamma_2 := \omega_3+\gamma_1=i \tau+\gamma_1=\left[ i \tau, i\tau  + 2 \omega_1 \right) ,
\end{equation}
defined on $\Omega$ and lifted to $\TT$.  These will be the two contours of orthogonality (as well as equivalent contours freely deformable on $\Omega\setminus \{0,a\}$ into $\gamma_1$ or $\gamma_2$) that we consider henceforth, see Figure~\ref{fig:gammas}.
Additionally, we assume that $a\in (\gamma_1\cup \gamma_2)\setminus \Gamma$, and that the leading coefficient of each $f_n$ in \eqref{orthbasis} is strictly positive. Notice that with these settings, all $b_j$ are real-valued on $\gamma_1 \cup \gamma_2$.

\begin{figure}[htb]
    \centering
\begin{tikzpicture}
    \coordinate (A) at (0,0);
    \coordinate (B) at (4,0);
    \coordinate (C) at (4,3);
    \coordinate (D) at (0,3);
    \draw[line width=2pt] (A) -- (B);
    \draw[dashed] (B) -- (C);
    \draw[dashed] (C) -- (D);
    \draw (D) -- (A);
    \coordinate (MidAD) at ($ (D)!0.5!(A) $); 
    \coordinate (MidBC) at ($ (C)!0.5!(B) $); 
    \draw[line width=2pt] (MidAD) -- (MidBC);
    \fill[red] (A) circle (3pt);
    \coordinate (BaseMid) at ($ (A)!0.5!(B) $);
    \node[above=1mm] at (BaseMid) {$\gamma_1$};
    \coordinate (RedLineMid) at ($ (MidAD)!0.5!(MidBC) $);
    \node[above=1mm] at (RedLineMid) {$\gamma_2$};
 \node[left=1mm] at (A) {$0$};
  \node[right=1mm] at (B) {$2\omega_1$};
  \node[left=1mm] at (MidAD) {$\omega_3$};
  \node[right=1mm] at (MidBC) {$\omega_3+2\omega_1$};
\end{tikzpicture}
\caption{Contours $\gamma_1$ and $\gamma_2$ on $\Omega$.}
    \label{fig:gammas}
\end{figure}
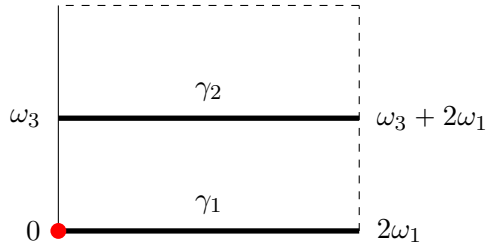

One of the main features of orthogonal polynomials on the real line (OPRL) is the property of their zeros: they are real, simple, and interlace with polynomials of the lower degree. In the rest of this section, we discuss the corresponding analogues. When discussing the interlacing of zeros of elliptic functions (i.e., meromorphic functions on $\TT$), we restrict our attention to the zeros in the fundamental parallelogram $\Omega$. This convention also applies to statements about the order of the zeros on a contour of $\TT$.

\begin{definition}[Interlacing] \label{def:interlacing}
Let $F, G$ be functions defined on the curve $\gamma \subset \C$, and let $z_1, \dots, z_n$ and $w_1, \dots, w_m$ be the zeros of $F$ and $G$ in $\gamma$, respectively, with account of their multiplicity. 
We say that $G$ \textbf{interlaces $F$ in $\gamma$} and denote it $F \preccurlyeq G$ in $\gamma$, if 
\begin{equation} \label{interlacing1}
    m=n \quad \text{and} \quad \Re z_1 \leq \Re w_1 \leq \Re z_2 \leq \Re w_2 \leq \cdots \leq  \Re z_n \leq \Re w_n,
\end{equation}
or if
\begin{equation} \label{interlacing2}
    m=n-1 \quad \text{and} \quad  \Re z_1 \leq \Re w_1 \leq \Re z_2 \leq \cdots \leq \Re z_{n-1} \leq \Re w_{n-1} \leq \Re z_n.
\end{equation}
Furthermore, we use the notation $F \prec$ in $\gamma$ when all inequalities in \eqref{interlacing1} or \eqref{interlacing2} are strict. 
\end{definition}

The following theorem states that the zeros of the $a$-EOPs mirror the properties of the zeros of OPRL.      
\begin{thm}\label{theo:zerosAndInterlacing}
    Let $W $ be a weight function on $ \Gamma \in \{\gamma_1, \gamma_2\}$, $a \in  (\gamma_1\cup  \gamma_2 ) \setminus  \Gamma  $, and $(f_n)$ be the sequence of elliptic orthogonal $a$-polynomials associated to $W$.
    \begin{enumerate}[(i)]
        \item For $\Gamma = \gamma_1$, $f_n$ has exactly $n$ simple zeros on $\gamma_1$. Additionally, it has a simple zero at $\gamma_2$ if, and only if, it has a pole at $a$. Moreover, $f_{n+1} \prec f_n$ in $\gamma_1$.
        \item For $\Gamma = \gamma_2$: 
            \begin{itemize}
                \item if $n$ is even, $f_n$ has exactly $n$ simple zeros on $\gamma_2$. Additionally, it has a (simple) zero on $\gamma_1$ if, and only if, it has a pole at $a$;
                \item if $n$ is odd, $f_n$ has exactly $n+1$ simple zeros on $\gamma_2$, and $f_n \prec  
 f_{n+1}$ in $\gamma_2$ or $f_{n+1} \prec  f_n$ in $\gamma_2$.
            \end{itemize}
        \end{enumerate}
       In particular, the zeros of $f_n$ are analytic functions of the parameter $a$.
    \end{thm}
    \begin{remark}
        Theorem \ref{theo:zerosAndInterlacing} is an extension of \cite[Theorem 2.3]{Bertola2022a}, where only the case $\Gamma = \gamma_2$ was considered. However, the interlacing property in this case is new.
    \end{remark}

The rest of this section is devoted to the proof of Theorem~\ref{theo:zerosAndInterlacing}. We start by studying the amount of zeros of $f_n$ on the support of orthogonality. 

    The next lemma and proposition are a simple generalization of \cite[Theorem 2.3]{Bertola2022a}:
				\begin{lemma}\label{lem:EvenAmountOfzeros}
                    Let $\Phi \in \mathcal L  (n\cdot 0 +a)$ for some $n\in \N$ be such that $\Phi(x)\in \R$ for $x\in \gamma_1\setminus\{a\}$. Then, the amount of zeros of $\Phi$ on $\gamma_2$, with account of multiplicity, is
                    \begin{enumerate}[(i)]
                        \item even, if $a\in \gamma_1$ or if $a\in \gamma_2$ but $\Phi$ is analytic at $z=a$;
                        \item odd, if $a\in \gamma_2$ and $\Phi$ has a pole at $z=a$.
                    \end{enumerate}
				\end{lemma}
				\begin{proof}
    Recall that for any elliptic function $\Phi$ with periods $2\omega_1>0$ and $2\omega_3=2i\tau$, $\tau>0$, we have that
    \begin{equation}\label{eq:EllipticFunctionsProperty}
                             \sum_{\substack{\Phi(p) = \infty \\ p \in \Omega}} p = \sum_{\substack{\Phi(z) = 0 \\ z \in \Omega}} z  , \mod \Lambda
                        \end{equation}
                        (meaning that the difference of the left and right-hand sides belongs to the lattice $\Lambda$), where all zeros and poles are enumerated with account of multiplicity. Taking the imaginary part, this identity is reduced to
                        \begin{equation}\label{eq:EvenzerosLemma}
                         \sum_{\substack{\Phi(p) = \infty \\ p \in \Omega}} \Im (p) = \sum_{\substack{\Phi(z) = 0 \\ z \in \Omega}} \Im(z)  , \mod 2\tau.
                        \end{equation}
If $\Phi \in \mathcal L  (n\cdot 0 +a) $, then 
$$
\sum_{\substack{\Phi(p) = \infty \\ p \in \Omega}} \Im (p) =\begin{cases}
    0 \mod 2\tau, & \text{if $ a\in \gamma_1$ or if $\Phi$ analytic at $a$,}\\
     \tau \mod 2\tau, & \text{if  $a\in \gamma_2$ and $\Phi$ has a pole at $a$.} 
\end{cases}
$$
Note that $z,w\in \Omega$ are such that $w =\overline{z} \mod \Lambda$, then $z\neq w \mod \Lambda$ if and only if $z,w\notin \gamma_1\cup \gamma_2$. In addition, for such a pair $z,w$, by assumption on $\Phi$, $\Phi(z)=0$ if and only if $\Phi(w)=0$. However, zeros of $\Phi$ on $\gamma_1$, or pairs of zeros $(z,w)$ not on $\gamma_1\cup \gamma_2$ do not contribute to the sum in 
$$
\sum_{\substack{\Phi(z) = 0 \\ z \in \Omega}} \Im(z) ,
$$
so that
$$
\sum_{\substack{\Phi(z) = 0 \\ z \in \Omega}} \Im(z) = \sum_{\substack{\Phi(z) = 0 \\ z \in \gamma_2}} \Im(z) .
$$
This yields that
$$
\sum_{\substack{\Phi(z) = 0 \\ z \in \gamma_2}} \Im(z)=\begin{cases}
    0 \mod 2\tau, & \text{if $ a\in \gamma_1$ or $\Phi$ analytic at $a$,}\\
     \tau \mod 2\tau, & \text{if  $a\in \gamma_2$ and $\Phi$ has a pole at $a$,} 
\end{cases}
$$
which is equivalent to the statement of the lemma.
				\end{proof}		

We now investigate the location of the zeros. In the following statement, we denote by $[\cdot]$ the integer part of a real number.		
   \begin{prop}\label{prop:NumberOfzeros}
				Let $f_n$ be the $n$-th orthogonal elliptic $a$-polynomial with respect to $W$, as in \eqref{orthbasis}. 
						\begin{enumerate}[(i)]
                \item If $\Gamma=\gamma_2$ and $a \in \gamma_1$ then all zeros of $f_n$ are simple, and $2[(n+1)/2]$ of them belong to $ \gamma_2$. Moreover,
                \begin{itemize}
                 \item if $n$ is odd, then $f_n$ has no other zeros on $\Omega$ and has a pole at $a$.
				\item if $n$ is even, $f_n$ can have at most one additional zero, which is simple and on $\gamma_1$, if and only if it has a pole at $a$.	
                \end{itemize}
                \item If $\Gamma=\gamma_1$ and $a \in \gamma_2$ then all zeros of $f_n$ are simple, and $n$ of them belong to $ \gamma_1$. In addition, $f_n$ has a zero on $\gamma_2$ if and only if it has a pole at $a$.
\end{enumerate}							
				\end{prop}
				\begin{proof}
The proof mimics the classical arguments for OPRL but needs an additional fact about the existence of a particular elliptic function. Let us prove \textit{(i)}; the statement of \textit{(ii)} is established in the same fashion.

Assume that $\Gamma=\gamma_2$, $a\in \gamma_1$, and let $  z_1, \dots , z_m$ be the distinct zeros of $f_n$ on $\gamma_2$ with odd multiplicities. Since $f_n\in \mathcal L (n\cdot 0 +a)$, we have that $m\le n+1$. Inspired by \eqref{eq:EllipticFunctionsProperty}, let $y\in \Omega$ be such that
$$
 y + \sum_{k=1}^m z_k = a= m\cdot 0 + a \mod \Lambda.
$$
By Lemma \ref{lem:EvenAmountOfzeros}, $m$ is even, which shows that $y\in \gamma_1$. By Abel's Theorem \cite[Chapter V, Theorem 2.8]{Miranda}, there exists an elliptic function $\Phi$ with simple zeros at $z_1,\dots, z_m$ and at $y$, a simple pole at $a$, and a pole of multiplicity $m$ at $0$. Moreover, $\Phi(z)\in \R$ for $z\in \gamma_1\cup \gamma_2$. Indeed, since $\Phi \in \mathcal L(m\cdot 0 + a)$, we can write $\Phi = \sum_{j=0}^{m} \lambda_j b_j$, and it is sufficient to show that all $\lambda_j\in \R$. The vector $\lambda^T=(\lambda_1, \dots, \lambda_m)$ is a non-trivial solution of the homogeneous system $A\lambda = 0$, with $A=\left( b_j(z_i)\right)_{i,j=0,1,\dots,m}\in \R^{(m+1)\times (m+1)}$, and thus, we can always choose $\lambda\in \R^m$. 

Since the product $\Phi f_n$ is real-valued and has constant sign on $\gamma_2$, the positivity of $W$ implies that
     \[
         \int_{\gamma_2} \Phi(z) f_n(z) W(z) \dd z \neq 0.
     \]
This contradicts the orthogonality of $f_n$, unless $m\ge n$. Thus, $n\le m\le n+1$. Since $m$ is even, we conclude that $m$ is equal to the only even integer in the set $\{n, n+1\}$. In other words, if $n$ is odd, then $m=n+1$; it means that all zeros $z_1, \dots, z_m$ are simple and that $f_n$ must have a pole at $z=a$. 
On the other hand, if $n$ is even, then $m=n$. Once again, all zeros $z_1, \dots, z_m$ must be simple, but $f_n$ still can have another (simple) zero, this time on $\gamma_1$, which occurs if and only if $f_n$ has a pole at $z=a$.
		\end{proof}
\begin{remark}
    On Section \ref{sec:OPRL-EOP} we present choices of $W$ and $a$ for which the sequence $f_{2n}$ never have a pole at $a$.
\end{remark}

 To finish the proof of Theorem \ref{theo:zerosAndInterlacing}, it only remains to verify the interlacing property. The lack of the three-term recurrence relation satisfied by $\{f_n\}$ impedes the application of the standard arguments based on the Christoffel-Darboux kernel (see, e.g., \cite[Theorem 2.2.3]{Ismail2005}). 
        \begin{prop}\label{prop:Interlacing}
Let $f_n$, $n\ge 1$, be the $n$-th orthogonal elliptic $a$-polynomial with respect to $W$, as in \eqref{orthbasis}. 
\begin{enumerate}[(i)]
    \item  If $\Gamma = \gamma_1$ and $a \in \gamma_2$, then $f_{n+1} \prec f_n$ in $\gamma_1$;
    \item  If $\Gamma = \gamma_2$, $a \in \gamma_1$ and $n$ is odd, then $f_n \prec f_{n+1}$ in $\gamma_2$ or $f_{n+1} \prec f_n$ in $\gamma_2$. 
        \end{enumerate}
        \end{prop}
        \begin{proof}
       Consider a linear combination $G  = A f_n  + B f_{n+1} $, with $A, B \in \R$, $A^2 +B^2\neq 0$. Since $G\in \mathcal L ((n+1)\cdot 0 +a)$, the total number of its zeros in $\Omega$, with the account of multiplicity, is $\le n+2$. Moreover, by equation \eqref{eq:EllipticFunctionsProperty},
\begin{equation}
    \label{eq:EllipticFunctionsProperty1}
      \sum_{\substack{G(z) = 0 \\ z \in \Omega}} z =    \varepsilon a , \mod \Lambda,
\end{equation}
where $\varepsilon=1$ if $G$ has a pole at $a$, and $\varepsilon=0$ otherwise.

We address the case \textit{(i)} first. By Lemma~\ref{lem:EvenAmountOfzeros}, \textit{(ii)}, the number of zeros of $G$ in $\Omega$ is exactly $n+2$ if and only if $G$ has a pole at $a$, meaning that $\varepsilon=1$ in \eqref{eq:EllipticFunctionsProperty1}. This shows that all zeros of $G$ cannot be on $\gamma_1$, and thus, the number of zeros on $\Gamma$, with account of multiplicity, is $\le n+1$. 

Let $z_1, \dots, z_m$ be the distinct zeros of odd order of $G  = A f_n  + B f_{n+1} $ on $\Gamma$. Reasoning as in the proof of Proposition~\ref{prop:NumberOfzeros}, taking $y\in \Omega$ such that
\begin{equation}
    \label{constructionAbel}
 y + \sum_{k=1}^m z_k = a= m\cdot 0 + a \mod \Lambda,
\end{equation}
and invoking Abel's Theorem, we can claim the existence of an elliptic function $\Phi\in \mathcal L (m \cdot 0 +a)$ with simple zeros at $z_1,\dots, z_m$ and at $y$, a simple pole at $a$, and a pole of multiplicity $m$ at $0$. Moreover, $y\in \gamma_2$, $\Phi$ is real-valued on  $\Gamma$, and $z_1, \dots, z_m$ are the only (simple) zeros of $\Phi$ on $\Gamma$. Thus, 
 \[
         \int_{\Gamma} \Phi(z) G(z) W(z) \, \dd z \neq 0.
     \]
Since $G$ is orthogonal to $\mathcal L ((n-1) \cdot 0 +a)$, we conclude that $m\ge n$. This implies that all zeros of $G$ on $\Gamma$ are simple (otherwise, the total number of zeros would be at least $n+2$). 

We aim at a similar conclusion in the case \textit{(ii)}. Now Lemma~\ref{lem:EvenAmountOfzeros}, \textit{(i)}, yields that $G$ must have an even number of zeros of $G$ on $\gamma_2$, and hence, there are at most $ n+1$ on $\Gamma$. 
Moreover, if $z_1, \dots, z_m$ are the distinct zeros of odd multiplicity of $G  = A f_n  + B f_{n+1} $ on $\Gamma$, then $m$ is even, since otherwise the total amount of zeroes of $G$ on $\Gamma$ would be odd. We take $y\in \Omega$ satisfying \eqref{constructionAbel} and construct $\Phi$ as above, except that now $y \in \gamma_1$. Thus, $z_1, \dots, z_m$ are the only (simple) zeros of $\Phi$ on $\Gamma$. Reasoning as above, we conclude that $m\ge n$, and once again, all zeros of $G$ on $\Gamma$ are simple. 
    
    In both cases \textit{(i)} and \textit{(ii)}, we have shown that $G$ has either $n$ or $n+1$ zeros on $\Gamma$, all simple. This implies that the linear system
    \[
                \begin{pmatrix}
                f_{n}(z)  &   f_{n+1}(z)\\
                f_{n}'(z) &   f_{n+1}'(z)
                \end{pmatrix}
            \begin{pmatrix}
                A \\ B
            \end{pmatrix}
            = 
            \begin{pmatrix}
                0 \\ 0
            \end{pmatrix}
            \]
            admits only the trivial solution for $A=B=0$. Thus, the determinant of the matrix is non-vanishing:
            \begin{equation}
            \label{eq:InterlacingEq}
             f_{n+1}(z) f_n'(z) - f_n(z) f_{n+1}'(z) \neq 0, \quad z \in \Gamma.
            \end{equation}
            In particular, the left-hand side does not change sign on $\Gamma$.

            Finally, we can proceed with the proof of interlacing. Let $x$ and $y$ be two consecutive zeros of $f_n$ on $\Gamma$. All the zeros of $f_n$ are simple (Proposition~\ref{prop:NumberOfzeros}), and thus $f_n'(x) f_n'(y)<0$. From \eqref{eq:InterlacingEq},  $f_{n+1}(x) f_{n+1}(y)<0$, which implies that there is a zero of $f_{n+1}$ on $(x,y)$.  
            This proves the interlacing result: $f_n \prec f_{n+1}$ or $f_{n+1} \prec f_n$.

           In the case \textit{(i)}, when $0\in \Gamma$, we can take $z \to 0$ to conclude that  
           \[
            f_{n+1}(z) f_n'(z) - f_n(z) f_{n+1}'(z) > 0, \quad z \in \Gamma\setminus\{0\}.
            \]
            Let $x$ be the left-most zero of $f_n$ on $\gamma_1$; since the leading coefficient of $f_n$ is strictly positive, $  f_n'(x) < 0$, and equation \eqref{eq:InterlacingEq} shows that $  f_{n+1}(x) < 0$. As $\lim_{z \to 0^+} f_{n+1}(z) =  + \infty$, there exists $x' < x$ such that $f_{n+1}(x) = 0$. A similar argument shows that if $y$ is the right-most zero of $f_n$ then there exists $y' > y$ such that $f_{n+1}(y') = 0$. This finishes the proof for \textit{(i)}.

            \end{proof}

   \section{A correspondence between OPRL and EOP}\label{sec:OPRL-EOP}

In this section, we explore the possible correspondence between the orthogonal $a$-polynomials on $\mathcal T$, and families of OPRL. 
 
 \subsection{A totally symmetric case}\label{sec:TotalSymmetry} \hfill

Starting with a positive weight function $w$ on a compact interval of the real-line, we discuss next how to lift its corresponding family of orthogonal polynomials to a family of orthogonal elliptic $a$-polynomials, corresponding to a specific torus $\TT$, for which the orthogonality weight $W$ is an even function and $a$ is a half-period. This corresponds to the maximum amount of symmetry possible over $W$ and $a$. We do it with the goal of illustrating the procedure in the simplest situation. 

Given two real numbers  $e_2, e_3$ satisfying
\begin{equation}
    \label{condtone}
e_3<0, \quad e_3<e_2<|e_3|/2,  
\end{equation}
define $e_1=-e_2-e_3$; notice that with our assumptions, the relation
$$
  e_3<e_2<e_1 
$$
holds. Let $w$ be a weight that is $>0$ almost everywhere on the real interval $[e_3,e_2]$, and let 
\begin{equation}
    \label{wtilde}
 \widetilde{w}(x) \defeq \frac{(x-e_3)(e_2-x)}{e_1 - x}\, w(x), \quad x\in [e_3 , e_2]
 \end{equation}
 be its rational modification. We denote by $P_n(\cdot; w)$ and $P_n(\cdot; \widetilde w)$ the algebraic monic  polynomial of degree $n$, orthogonal on $[e_3,e_2]$ with respect to $w$ and $\widetilde w$, respectively.

By formula~\eqref{eq:wpODE}, the triplet $(e_1, e_2, e_3)$ determines uniquely the Weierstrass function $\wp = \wp (\cdot; \Lambda)$, where the lattice $\Lambda$ is rectangular, meaning that we can take the half-periods $ \omega_1 >0$ and $ \omega_3 = i \tau$, $\tau>0$ (see \eqref{formulaperiods}).

       \begin{thm}\label{teo:lifting_symmetric}
            Under the assumptions above, let $a = \omega_1\in \gamma_1$. Then the function $F_n\in \mathcal L(n\cdot 0 +a)$, $n\in \Z_{\ge 0}$, defined by
            \[
            F_n(z) := \begin{cases}
                P_j(\wp(z); w), \quad n = 2j,\; j\geq 0, \\
                - \dfrac{1}{2}\dfrac{\wp'(z) }{\wp(z) - e_1} P_j(\wp(z); \widetilde w), \quad n = 2j+1,\; j\geq 0,
            \end{cases}
        \]
 is the (monic) $a$-polynomial of the polynomial degree $n$ on the torus $\mathcal T=\C/ \Lambda$, orthogonal with respect to the even weight $W$ on $\gamma_2$, uniquely defined by
    \begin{equation}\label{eq:Weight_of_Lift}
        W(z) := \begin{cases}
        \dfrac{1}{2}\, w(\wp(z)) \wp'(z), & z = \omega_3 + t\omega_1, \quad t\in [0,1],\\[2mm]
        W(2\omega_2-z), & z = \omega_3 + t\omega_1, \quad t\in [1,2].
    \end{cases}
    \end{equation}
        \end{thm}

Notice that $F_n$ has a pole at $a$ if, and only if, $n$ is odd. When $n$ is even, $F_n$ coincide with the even elliptic polynomials with symmetric orthogonality weight constructed in \cite{Desiraju2024}.
        \begin{proof}
Under our assumptions on the periods, function $\wp(\omega_3 + t\omega_1)$ is strictly increasing on $t\in (0,1)$, taking all values in the interval $(e_3, e_2)$, see e.g., \cite[$\mathsection 22$]{Markushevich}. By symmetry,  $\wp(\omega_3 + t\omega_1)$ is strictly decreasing on $t\in (1,2)$; see Figure \ref{fig:wp}, right.

    \begin{figure}[ht]
\centering
\begin{subfigure}{.5\textwidth}
  \centering
  \includegraphics[scale=0.6]{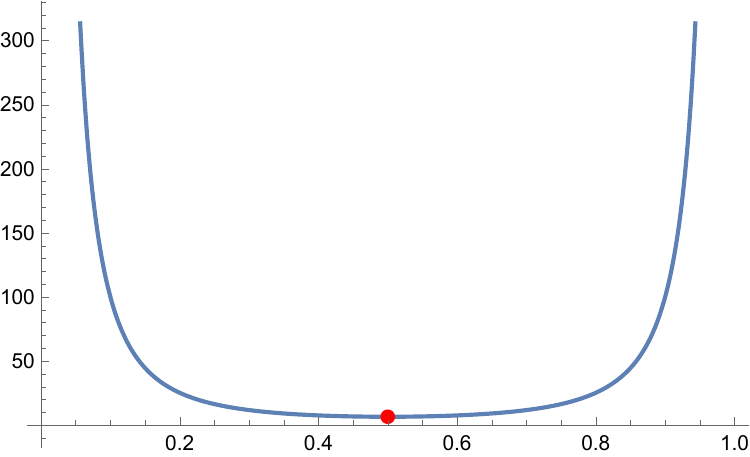}
  \caption{$\wp(t\omega_1)$, $0 \leq t \leq 2$.}
  \label{fig:sub1}
\end{subfigure}%
\begin{subfigure}{.5\textwidth}
  \centering
  \includegraphics[scale=0.6]{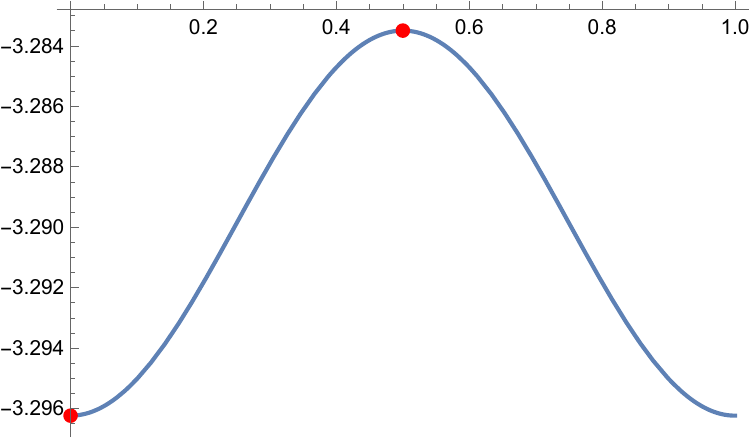}
  \caption{$\wp(\omega_3 + t\omega_1)$, $0 \leq t \leq 2$.}
  \label{fig:sub2}
\end{subfigure}
\caption{Plots of $\wp(z)\in \R$ for $2 \omega_1 = 1$ and $2\omega_3 =  3i/2$ on the contours $\gamma_1$ (left) and $\gamma_2$ (right). Here $e_1 \approx 6.57974$, $e_2 \approx -3.29624$, and $e_3 \approx -3.2835$.}
\label{fig:wp}
\end{figure}

In a similar fashion, $\wp(t\omega_1)$ is strictly decreasing on $t\in (0,1)$, taking all values in the interval $(e_1, +\infty)$. By symmetry,  $\wp(  t\omega_1)$ is strictly increasing on $t\in (1,2)$ (see Figure \ref{fig:wp}, left).

Additionally, since $\wp\left( \omega_j \right) = e_j$ and $\wp'\left(\omega_j \right) = 0$ (see Equations \eqref{eq:wpODE} and \eqref{defE}), in this case $b_1$ from \eqref{defB1} is an odd function that takes the form
\begin{equation}
    \label{b1symm}
                b_1(z) 
                =  - \frac{1}{2}\frac{\wp'(z) }{\wp(z) - e_1}.
\end{equation}

            It is straightforward to check that $F_n$ is a monic elliptic $a$-polynomial of degree $n$. We only need to verify the orthogonality conditions. Note that $F_n$ is an even (resp. odd) function if, and only if, $n$ is even (resp. odd).
            
            Since $W$ is an even weight, for $n,m\in \Z_{\ge 0}$, 
        \[
            \int_\Gamma F_{2n}(z) F_{2m+1}(z) W(z) \, \dd z = 0.
        \] 
        For $n = 2j\in \Z_{\ge 0}$, $m=2l\in \Z_{\ge 0}$, $m\neq n$,
        \begin{align*}
            \int_\Gamma F_n(z) F_m(z) W(z) \, \dd z &=
            2 \int_{\omega_3}^{\omega_3+\omega_1} P_j(\wp(z); w) P_l(\wp(z); w) W(z) \, \dd z
           \\
       & =  \int_{\omega_3}^{\omega_3+\omega_1} P_j(\wp(z); w) P_l(\wp(z); w) w(\wp(z)) \wp'(z) \, \dd z \\
       & =\int_{e_3}^{e_2} P_j(x;w) P_l(x;w) w(x) \, \dd x = 0.
         \end{align*}
        Analogously, using \eqref{eq:wpODE} we get that for $n = 2j+1\in \N$, $m = 2l+1\in \N$, $m\neq n$,
        \[
            \int_\Gamma F_n(z) F_m(z) W(z) \, \dd z = 
            \frac{1}{4} \int_{e_3}^{e_2} P_j(x;\widetilde w) P_l(x;\widetilde w) \frac{(x-e_2)(x-e_3)}{(x-e_1)} w(x) \, \dd x = 0.
        \]
        \end{proof}

Theorem~\ref{teo:lifting_symmetric} shows how to lift the orthogonal polynomials from the real line to a torus $\mathcal T$, corresponding to a rectangular lattice. The obtained construction had the feature that
    \begin{equation}\label{eq:symmetry-assumptions}
         W \text{ is an even function on } \Gamma = \gamma_2, \qquad a = \omega_1 \text{ (half period). } 
    \end{equation}
    Note that for every positive even weight function $W$ defined on $\gamma_2$, equation \eqref{eq:Weight_of_Lift} holds with
    \[
        w(x) = \frac{W(\wp^{-1}(x))}{\sqrt{ (e_1 - x)(e_2-x)(x-e_3)}} > 0, \quad e_3 < x < e_2,
    \]
    therefore Theorem \ref{teo:lifting_symmetric} can be applied to $w$, characterizing the sequence of $a$-EOP with respect to $W$ in terms of $P_n(x,w)$ and $P_n(x,\tilde{w})$, where
    \[
        \tilde{w}(x) = \sqrt{\frac{(x-e_3)(e_2-x)}{(e_1 - x)^3}} W(\wp^{-1}(x)).
    \]
   The following identities also holds true for the monic orthogonal polynomials associated to $w$ and $\tilde{w}$ defined by \ref{wtilde},
     \[
        P_j(x,w) = F_{2j}(\wp^{-1}(x)), \quad P_j(x, \tilde{w}) = \sqrt{\frac{e_1 - x}{(e_2 - x)(x - e_3)}} F_{2j+1}(\wp^{-1}(x)), \quad x \in [e_3,e_2],
     \]
     where $F_n$ is the $n$-th monic $a$-EOP associated to $W$, with $a = \omega_1$.
     
In other words, the construction of Theorem~\ref{teo:lifting_symmetric} is reversible, and establishes a bijection between $a$-EOPs with respect to the weight $W$ satisfying \eqref{eq:symmetry-assumptions}, and two related sequences of OPRL.

\begin{example} \label{exampleJacobi1}
Consider the torus $\TT$ defined by the square lattice 
$$
\Lambda = 2\omega_1   \Z + 2i\omega_1    \Z, \quad \omega_1=\frac{32 \pi }{\Gamma \left( \frac{1}{4}\right)^4},
$$
for which the constants \eqref{defE} can be found explicitly, see \cite[$\mathsection 23.5$(iii)]{NIST:DLMF}:
    \[
        e_2 = 0, \quad e_1 = 1= - e_3,
    \]
    as well as the Weierstrass invariants,
    $$
    g_2= 4 \omega_1^{-6}, \quad g_3=0.
    $$

We illustrate the application of Theorem \ref{teo:lifting_symmetric} by an explicit construction of elliptic orthogonal $a$-polynomials with respect to a weight on $\mathcal T$ in terms of the Jacobi family of OPRL. 

Let $\wp(z)=\wp(z;\Lambda)$, and consider the weight
  \begin{equation}\label{eq:Example_W}
        \left| \wp(z)\right|^{\alpha + \frac{1}{2}} (\wp(z)+1)^{\beta + \frac{1}{2}}  (1 - \wp(z))^{\frac{1}{2}}>0, \quad z \in \gamma_2,
    \end{equation}
where $\alpha, \beta>-1$.

If $ P_j^{(\alpha, \beta)} $  denotes the $j$-th monic Jacobi polynomial, then 
\[
        P_j(x;w) = 2^{-j} P_j^{(\alpha,\beta)}(2x + 1),
    \]
is the monic polynomial of degree $j$, orthogonal on $[-1,0]$ with respect to the weight
        \[
            w(x) = |x|^\alpha(1+x)^\beta.
        \]

According to Theorem \ref{teo:lifting_symmetric}, for even values of $n$,
    \[
        F_n(z) = \frac{1}{2^j} P_j^{\alpha, \beta} (2\wp(z) + 1), \quad n = 2j,
    \]
    is the $n$-th elliptic orthogonal $a$-polynomial, $a = \omega_1$, with respect to the weight \eqref{eq:Example_W}.
  
On the other hand, the corresponding $\widetilde{w}$ function defined in \eqref{wtilde} is 
    $$
    \widetilde w(x):=\frac{|x|^{\alpha+1} (x+1)^{\beta+1}}{1-x}, \quad x\in [-1,0],
    $$
   which is a rational deformation of a Jacobi weight on $[-1,0]$. Define
   \begin{equation}
       \label{defkappan}
    \lambda_n=    \lambda_n(\alpha,\beta) \defeq 
            \int_{-1}^1 \frac{P_n^{(\alpha+1,\beta+1)}(s)}{3-s} (1-s)^{\alpha+1} (1+s)^{\beta+1} \, \dd s.
 \end{equation}
It is easy to directly check (see e.g.~\cite[Theorem 2.7.2]{Ismail2005}) that in this case,
   \[
        \begin{split}
            P_j \left( x, \widetilde w\right) 
            &= \frac{1}{2^{j} \lambda_{j-1}} \left[
            \lambda_{j-1} P_j^{(\alpha +1, \beta + 1)}(2x+1)- \lambda_j P_{j-1}^{(\alpha +1, \beta + 1)}(2x+1)  \right] .
        \end{split}
    \]
    Therefore, by Theorem \ref{teo:lifting_symmetric}, for odd values of $n$,
\begin{align*}
     F_n(z) = & \frac{1}{2^{j+1}\lambda_{j-1}}\frac{\wp'(z)}{1 - \wp(z)} \\
     & \times \left[ \lambda_{j-1} P_j^{(\alpha +1, \beta + 1)}(2 \wp(z) + 1) - \lambda_j P_{j-1}^{(\alpha +1, \beta + 1)}(2 \wp(z) + 1) \right], \quad \text{if }n = 2j+1
\end{align*}
    is the $n$-th elliptic $a$-orthogonal polynomial with respect to $W$ in \eqref{eq:Example_W}.
\end{example}

\begin{example} \label{exampleJacobi2}
Let $\mathcal T$ be the torus of Example~\ref{exampleJacobi1}, but instead of the weight $W$ as in \eqref{eq:Example_W}, we consider
 \begin{equation}
     \label{eq:Example_V}
        V(x) = 
        \left|\wp(z)\right|^{\alpha - \frac{1}{2}} (\wp(z) + 1)^{\beta - \frac{1}{2}} (1 - \wp(z))^{\frac{3}{2}} , \quad z \in \gamma_2.
 \end{equation}

Take
    \[
        v(x) =  (1-x)(1+x)^{\beta - 1} |x|^{\alpha -1}, \quad 
         \alpha, \beta > 0, \quad x \in [-1,0].
    \]
    and the corresponding weight $\widetilde{v}$ (see \eqref{wtilde}),
    \[
         \widetilde v(x) = |x|^\alpha (1+x)^\beta , \quad x \in [-1,0].
    \]
     Note that $v$ is a polynomial (Christoffel) deformation of the Jacobi weight on $[-1,0]$ with parameters $\beta-1$ and $\alpha-1$. 
    
    Then, by Theorem \ref{teo:lifting_symmetric}, for odd values of $n$,
    \[
        G_n(z) =  \frac{1}{2^{j+1}} \frac{\wp'(z)}{1 - \wp(z)} P_j^{(\alpha, \beta)} (2 \wp(z) + 1), \quad n = 2j+1,
    \]
    is the $n$-th elliptic orthogonal $a$-polynomial, $a = \omega_1$, with respect to the weight \eqref{eq:Example_V}. 
   
On the other hand, \cite[Theorem 2.7.1]{Ismail2005} and direct computation shows that  
    \[
        P_j(x,v) = \frac{K_j^{(\alpha-1, \beta-1)}(2x+1,3)}{2^{j} P_j^{(\alpha-1,\beta-1)}(3)},
    \]
where $K_j^{(\alpha,\beta)}$ is the (scaled) Christoffel-Darboux kernel,
    \begin{equation}\label{eq:cdlikekernel}
        K_j^{(\alpha,\beta)}(x,y) 
            \defeq \frac{P_{j+1}^{(\alpha, \beta)}(x) P_j^{(\alpha, \beta)}(y) - P_{j+1}^{(\alpha, \beta)}(y) P_j^{(\alpha, \beta)}(x)}{x-y}. 
    \end{equation}
    
    By Theorem \ref{teo:lifting_symmetric}, for even values of $n$,
    \[
        G_n(z) = \frac{K_j^{(\alpha-1, \beta-1)}(2\wp(z)+1,3)}{2^{j} P_j^{(\alpha-1,\beta-1)}(3)} , \quad n = 2j,
    \]
    is the $n$-th elliptic orthogonal $a$-polynomial associated to the weight $V$.

\end{example}

    Interestingly, the interlacing of the $a$-EOPs obtained in Theorem \ref{theo:zerosAndInterlacing} has a consequence for the OPRL sequence with respect to $w$ and $\tilde{w}$. This is not expected, as no such result holds for general rational deformations of measures.
    \begin{cor}\label{cor:interlacing}
        Set $\,\widehat \gamma_2 = \{ \omega_3 + t \omega_1; t \in [0,1]\}$. Then, with the notations of Definition~\ref{def:interlacing} and Theorem \ref{teo:lifting_symmetric}, $F_n \prec F_{n+2}$ in $\widehat{\gamma}_2$.
        Additionally, for the orthogonal polynomials with respect to $w$ and its rational deformation $\widetilde{w}$ as in \eqref{wtilde}, the following interlacing property holds:
        \[
        P_n(\cdot, w) \prec P_{n-1}(\cdot, \widetilde w) \text{ in } [e_3, e_2].
        \]
    \end{cor}  
        \begin{proof}
            Assume that $n = 2m$ and let $z_1^{(n)}, \ldots, z_m^{(n)}$ be the zeros of $F_n$ on $\widehat \gamma_2$. By Theorem \ref{teo:lifting_symmetric}, $\wp\left(z^{(n)}_1 \right), \ldots, \wp\left(z^{(n)}_k \right)$ are the zeros of $P_m(\cdot, w)$. Since $P_{m+1}(\cdot, w) \prec P_m(\cdot, w)$  in $[e_2,e_3]$ and $\wp$ is monotone on $\widehat{\gamma_2}$, we see that
            \[
                z_1^{(n+2)} < z_1^{(n)} < z_2^{(n+2)} < \cdots < z_{m}^{(n+2)} < z_m^{(n)} < z_{m+1}^{(n+2)},
            \]
            which establishes that $F_n \prec F_{n+2}$ in $\widehat \gamma_2$ in this case.  The proof for $n = 2m+1$ is analogous, by identifying the zeros of $f_n$ with those of $P_m(\cdot, \widetilde w)$.

            Similarly, the interlacing of $ P_n(\cdot, w)$ and $ P_{n-1}(\cdot, \widetilde w)$ in $[e_3,e_2]$ is  consequence of the identification of their zeros with the zeros of $f_{2n-1}$ and $f_{2n}$, and their corresponding interlacing given by Proposition \ref{prop:Interlacing}.
        \end{proof}

    Applying Corollary \ref{cor:interlacing} to the construction of Examples~\ref{exampleJacobi1} and \ref{exampleJacobi2} we obtain some interlacing properties for Jacobi polynomials:
    \begin{cor} \label{cor:Jacobi}
        Let $P_n^{(\alpha,\beta)}$ be the $n$-th monic Jacobi polynomial, $K_j^{(\alpha, \beta)}$ defined by \eqref{eq:cdlikekernel} and $\lambda_n = \lambda_n(\alpha, \beta)$ defined by \eqref{defkappan}. The polynomials
        \[
        \begin{split}
            R_{n-1}(x) 
                &\defeq P_{n-1}^{(\alpha +1, \beta+1)}(x) \lambda_{n-2} - P_{n-2}^{(\alpha +1, \beta+1)}(x) \lambda_{n-1}, \\
            S_n(x) 
                & \defeq K_{n+1}^{(\alpha-1, \beta-1)}(x,3),
            \end{split}
        \]
        satisfy the following interlacing properties:
        \[
            \begin{split}
                S_n &\prec P_n^{(\alpha, \beta)} \text{ in } [-1,1], \quad \alpha, \beta > 0, \\
                P_n^{(\alpha, \beta)} &\prec R_{n-1} \text{ in } [-1,1], \quad \alpha, \beta > -1.
            \end{split}
        \]

    \end{cor}

    \begin{remark}
  Let us briefly put Corollary~\ref{cor:Jacobi} into the context of the literature on classical orthogonal polynomials.
        
        Note that 
        $R_{n-1}$ and $S_n$ are  polynomials of degree $n-1$ and $n$, respectively, orthogonal with respect to the weights
        \[
            \omega(x) = \frac{(x-1)^{\alpha +1} (x+1)^{\beta + 1}}{3-x}, \quad 
            \widetilde{\omega}(x) = (3-x)(1+x)^{\beta - 1} (x-1)^{\alpha - 1}, \quad x \in [-1,1].
        \]
    Theorem 2.3 in  \cite{DriverJordaanMbuyi2008} establishes the following interlacing between Jacobi polynomials of different degrees,
        \[
            P_n^{(\alpha, \beta)} \prec P_{n-1}^{(\alpha +t, \beta + t')}, \quad t, t' \in (0,2], \quad  \alpha, \beta > -1.
        \]
        We can see the Jacobi weight with parameters $(\alpha+1, \beta+1)$ as a Christoffel deformation of $\omega(x)$ and therefore 
        Theorem~$2$ of 
        \cite {DimitrovIsmailRafaeli2013} 
        gives the interlacing
        \[
            P^{(\alpha + 1, \beta + 1)}_{n-1} \prec R_{n-1}.
        \]
        Corollary \ref{cor:Jacobi} establishes that the chain of interlacing 
        \[
            P_n^{(\alpha, \beta)} \prec P_{n-1}^{(\alpha +1, \beta+1)} \prec R_{n-1}
        \]
        is transitive, fact that apparently is new. 
        
        Similarly, we can see $\widetilde{\omega}$ as a Christoffel deformation of the Jacobi weight with parameters $(\alpha -1, \beta-1)$ and, by Theorem 2 from \cite {DimitrovIsmailRafaeli2013} to $\widetilde{\omega}$,
        \[
            S_n(x) \prec P_n^{(\alpha -1, \beta-1)}(x).
        \]
        The simultaneous increase or decrease of parameters $\alpha$ and $\beta$ can break the interlacing between Jacobi polynomials of the same degree, so there is no expected interlacing between $P_n^{(\alpha-1, \beta-1)}$ and $P_n^{(\alpha, \beta)}$. Corollary~\ref{cor:Jacobi} establishes that $S_n$ is a common interlacing polynomial to both $P_n^{(\alpha, \beta)}$ and $P_n^{(\alpha -1, \beta-1)}$, something that we could not find explicitly stated in the literature.        
    \end{remark}

    \subsection{Decomposition of EOPs and orthogonality}\label{sec:Decomposition} \hfill

    In Subsection \ref{sec:TotalSymmetry}, we proved that when $W$ is an even weight and $a$ is a half-period, the sequence of $a$-EOPs splits into two sequences of orthogonal polynomials on the real-line given in terms of the weight $w(x) = \frac{W(z)}{|\wp'(z)|}$ where $x = \wp(z)$. We now drop the symmetries assumed on $W$ and $a$.
    
   We start with a general decomposition lemma holding for for any $a$-polynomial and show that, in the case of $F_n$, the polynomials involved satisfy certain orthogonality conditions.

    \begin{lemma}\label{lemma:Decomposition}
        Let $f\in \mathcal L (n\cdot 0+a)$ be an elliptic $a$-polynomial of polynomial degree exactly $n$. There exist unique algebraic polynomials $p_{n,1}$ and $p_{n,2}$, with $
            \deg p_{n,1} \leq \floor*{\frac{n}{2}}$,
            $\deg p_{n,2} \leq \floor*{\frac{n-1}{2}}
        $, such that $f$ can be written in the form %
\begin{equation}\label{eq:Decomposition_of_f_lemma}
            f(z) = p_{n,1}(\wp(z)) + b_1(z) \, p_{n,2}(\wp(z)) + \frac{\wp'(a)}{2} \, p_{n,3}(\wp(z)), 
            \quad p_{n,3}(z) \defeq \frac{p_{2}(z) - p_{2}(\wp(a))}{x - \wp(a)}.
        \end{equation}
    \end{lemma}
    \begin{proof}
    Let $n = 2m$ and write  $f(z) = \sum_{j=0}^n \lambda_{j} b_j(z)$, $\lambda_n \neq 0$. Using  that the second expression for the function $b_1$ in \eqref{defB1} is equivalent to the identity
    \begin{equation*}
        \wp'(z) = -2 \, b_1(z) [ \wp(z) - \wp(a)] - \wp'(a)
    \end{equation*}
     and the definition of $b_j$ in \eqref{basis1}, one gets \eqref{eq:Decomposition_of_f_lemma} with 
    \[
        p_{n,1}(x) = \sum_{j=0}^{m} \lambda_{2j} x^j, \quad p_{n,2}(x) = \lambda_1 + \sum_{j=0}^{m-2} \lambda_{2j + 3} [x - \wp(a)] x^j, \quad p_{n,3}(x) = \sum_{j=0}^{m-2} \lambda_{2j + 3} x^j.
    \]
    The proof for  $n = 2m+1$ is analogous.
    \end{proof}

 Recall that the Weierstrass $\wp$ function maps $\gamma_2$ twice on the interval $[e_3,e_2]$; more precisely, $\wp(\omega_3 + 2 t )$ is a strictly increasing function on $[0,1]$, mapping $[0,1]$ onto $[e_3, e_2]$. It is strictly decreasing on $[1,2]$, taking all values from $e_2$ to $e_3$ as $t$ transitions from $1$ to $2$. In this subsection, we denote by $z = \wp^{-1}(x)$ the inverse of $\wp^{-1}$ on the first half of the interval $\gamma_2$. In particular, $-z$ equals the inverse of $\wp$ restricted to the second half of $\gamma_2$. Moreover,
    \begin{alignat}{2}\label{eq:pprime_for_bj}
        b_1(z) &= - \frac{ \sqrt{\prod_{i=1}^3 (x-e_i)} + \frac{\wp'(a)}{2}}{x - \wp(a)}, \quad 
        &&b_1(-z) =  \frac{ \sqrt{\prod_{i=1}^3 (x-e_i)} - \frac{\wp'(a)}{2}}{x - \wp(a)}, \\
        b_{2l}(z) &= x^l , \quad 
        &&b_{2l+3}(z) = - x^l \sqrt{\prod_{i=1}^3 (x-e_i)}.
    \end{alignat}

 Let $F_n$ be the $n$-th monic elliptic orthogonal $a$-polynomial with respect to a general weight $W$ on $\gamma_2$ and $a\in \gamma_1$. We preserve the notation of $p_{n,1}$ and $p_{n,2}$ when applying Lemma \ref{lemma:Decomposition} to $F_n$. Additionally, define
  \begin{equation}\label{eq:defqn}
        q_n(x) \defeq p_{n,1}(x)( \wp(a) - x) + \frac{\wp'(a)}{2} p_{n,2}(\wp(a)),
    \end{equation}
as well as, for $j\in \{-1,0,1\}$,   the following real valued weight functions,
   \begin{equation}\label{defweights}
    w_{j}^{\pm}(x) \defeq \left(  \sqrt{\prod_{l=1}^3 (x-e_l)}\right)^j \frac{1}{\wp(a) - x} \frac{W(\wp^{-1}(x)) \pm W(-\wp^{-1}(x))   }{2}, \quad x \in [e_3,e_2].   
   \end{equation}
\begin{thm}\label{teo:General_Decomposition_f}
For $n\in \Z_{\ge 0}$, let $m=\floor*{\frac{n}{2}}$ and $k=n-2m\in \{0,1\}$.       With the notation above,
        the polynomials $p_{n,2}$ and $q_n$ satisfy the orthogonality conditions
\begin{equation}\label{eq:t1_mop_conditions_1}
        \int_{e_3}^{e_2} \left[ q_n(x)  w_{-1}^+(x) +  p_{n,2}(x)  w_0^-(x) \right] x^l \, \dd x = 0, \quad 0 \leq l < m + \frac{k}{2};
    \end{equation}
    \begin{equation}\label{eq:t1_mop_conditions_2}
        \begin{split}
            \int_{e_3}^{e_2} \left[ q_n(x)    w_0^-(x) + p_{n,2}(x)  w_1^+(x) \right] x^l  \, \dd x = 0, \quad 0 \leq l < m-2;
        \end{split}
    \end{equation}
    \begin{equation}\label{eq:t1_mop_conditions_3}
        \begin{split}
            &\int_{e_3}^{e_2} \left[ q_n(x) \left(  w_0^-(x)  + \frac{\wp'(a)}{2}  w_{-1}^+(x) \right) + p_{n,2}(x) \left( w_1^+(x)  + \frac{\wp'(a)}{2}  w_0^-(x)  \right) \right] \, \frac{\dd x}{\wp(a) - x} = 0.
            \end{split}
    \end{equation}
    \end{thm}

    \begin{remark}
 The results of Theorem~\ref{teo:General_Decomposition_f} hold even in the general case of $a \in \TT \setminus \left\{ \gamma_1 \cup \gamma_2\right\}$, when the weights $w_j^{\pm}$ are generally complex-valued. 
      
        Equations ~\eqref{eq:t1_mop_conditions_1}-\eqref{eq:t1_mop_conditions_3} define \textit{type II multiple orthogonality} conditions satisfied by  the vector of polynomials $(q_n, p_{n,2})$, see, e.g., \cite[$\mathsection 23$]{Ismail2005}.
    \end{remark}
    \begin{proof}
    The change of variable $x = \wp(z)$ and the definitions \eqref{eq:pprime_for_bj} give us
\begin{equation}\label{eq:orthogonality_for_deformation}
        \begin{split}
            0 
                &= \int_\gamma f_n(z) b_j(z) W(z) \, \dd z = \int_{\omega_3}^{\omega_3 + \omega_1}f_n(z) b_j(z) W(z) \, \dd z + \int_{\omega_3 + \omega_1}^{\omega_3 + 2 \omega_1} f_n(z) b_j(z) W(z) \, \dd z\\
                &= \int_{e_3}^{e_2} \left[p_{n,1}(x) + \frac{\wp'(a)}{2} p_{3,n}(x) \right] \left[\frac{b_j(z) W(z) + b_j(-z) W(-z)}{2 \sqrt{ \prod_{i=1}^3 (x-e_i)}} \right] \, \dd x \\
                &\phantom{=} + \int_{e_3}^{e_2} p_{n,2}(x) \left[ \frac{b_1(z) b_j(z) W(z) + b_1(-z) b_j(-z) W(-z)}{2 \sqrt{\prod_{i=1}^3 (x-e_i)}}\right] \, \dd x.
        \end{split}
    \end{equation}

Using again \eqref{eq:pprime_for_bj} with $j = 2l < n$, we get
        \begin{multline*}
            \int_{e_3}^{e_2} \left[p_{n,1}(x) + \frac{\wp'(a)}{2} p_{3,n}(x) \right] x^l \, [\wp(a) - x]w_{-1}^+(x) \, \dd x 
                = - \int_{e_3}^{e_2} p_{n,2}(x) x^l \left[ w_0^-(x) + \frac{\wp'(a)}{2} w_{-1}^+(x)  \right] \, \dd x.
        \end{multline*}
        
    By the definition of $q_n$ \eqref{eq:defqn} and the fact that $p_{3,n}(x) = \frac{  p_{n,2}(\wp(a)) - p_{n,2}(x)}{ \wp(a) - x}$,  we can rewrite the above as
    \[
        \int_{e_3}^{e_2} q_n(x) x^l w_{-1}^+(x) = - \int_{e_3}^{e_2} p_{n,2}(x) x^l w_0^-(x), \quad 0 \leq 2l < n,
    \]
    which is exactly \eqref{eq:t1_mop_conditions_1}.

   On the other hand, \eqref{eq:pprime_for_bj}  with $j = 2l + 3 < n$ shows that \eqref{eq:orthogonality_for_deformation} takes the form
    \begin{multline*}
        -\int_{e_3}^{e_2} \left[p_{n,1}(x) + \frac{\wp'(a)}{2} p_{3,n}(x) \right] x^l [\wp(a) - x] w_0^-(x) \, \dd x =
            \int_{e_3}^{e_2} p_{n,2}(x) x^l \left[ w_1^+(x) + \frac{\wp'(a)}{2} w_0^-(x) \right] \, \dd x.
    \end{multline*}

    Similarly to the previous case, we rewrite the identity above as
    \begin{equation*}
        \begin{split}
            \int_{e_3}^{e_2} q_n(x)  x^l  w_0^-(x) \, \dd x 
                &=  - \int_{e_3}^{e_2} p_{n,2}(x) x^l w_1^+(x)  \, \dd x, \quad 0 \leq 2l+3 < n,
        \end{split}
    \end{equation*}
    which is \eqref{eq:t1_mop_conditions_2}.
    
    Finally, for $j=1$, \eqref{eq:orthogonality_for_deformation} becomes
    \[
    \begin{split}
        &\int_{e_3}^{e_2} \left(p_{n,1}(x) + \frac{\wp'(a)}{2} p_{3,n}(x) \right) \left[ w_0^-(x) + \frac{\wp'(a)}{2} w_{-1}^+(x) \right] \, \dd x \\
        &= - \int_{e_3}^{e_2} \frac{p_{n,2}(x)}{\wp(a) - x} \left[ w_1^+(x) + \wp'(a) w_0^-(x) + \left( \frac{\wp'(a)}{2} \right)^2 w_{-1}^+(x) \right] \, \dd x.
    \end{split}
    \]
 This is equivalent to 
    \[
        \begin{split}
            &\int_{e_3}^{e_2} q_n(x) \left[ \frac{w_0^-(x)}{\wp(a) - x} + \frac{\wp'(a)}{2} \frac{w_{-1}^+(x)}{\wp(a) - x}\right] = - \int p_{n,2}(x) \left[ \frac{w_1^+(x)}{\wp(a) - x} + \frac{\wp'(a)}{2} \frac{w_0^-(x)}{\wp(a) - x}  \right] \, \dd x
        \end{split}
    \]
    that is, equation \eqref{eq:t1_mop_conditions_3}. This finishes the proof of Theorem \ref{teo:General_Decomposition_f}.
    \end{proof}
    
    Notice that when $W$ is even,  functions $w_j^{-}\equiv 0$, and equations \eqref{eq:t1_mop_conditions_1}-\eqref{eq:t1_mop_conditions_3} specialize to
\begin{align}
    \label{eq:Conditions_p1n_Weven}
        \int_{e_3}^{e_2} q_n(x) x^l  w_{-1}^+(x)  \, \dd x & = 0, \quad 0 \leq l < m + \frac{k}{2},
    \\
    \label{eq:Conditions_p2n_Weven}
        \int_{e_3}^{e_2} p_{n,2}(x) x^l w_1^{+}(x) \, \dd x &= 0, \quad 0 \leq l \leq m-2,
    \\
    \label{eq:ExtraCondition_Weven}
        \int_{e_3}^{e_2} \left[ q_n(x)  \frac{\wp'(a)}{2} \frac{w_{-1}^+(x)}{\wp(a) - x}
        + p_{n,2}(x) \frac{w_1^+(x)}{\wp(a) - x}  \right] \, \dd x &= 0.
    \end{align}

    Equation \eqref{eq:Conditions_p1n_Weven} defines orthogonality of $q_n$ with respect to $w_{-1}^+(x)$. By Lemma \ref{lemma:Decomposition},
    \begin{itemize}
        \item if $n=2m$ then $q_n$ has degree $m+1$ and orthogonality conditions in \eqref{eq:Conditions_p1n_Weven} are up to degree $m-1$. In other words, $q_n$ is quasi-orthogonal with respect to $w_{-1}^+(x)$.
        \item if $n = 2m+1$, then $q_n$ has degree $\leq m+1$ and \eqref{eq:Conditions_p1n_Weven} corresponds to orthogonality up to degree $m$. Now, if the degree is exactly $m+1$, then $q_n$ is orthogonal with respect to $w_{-1}^+(x)$. Otherwise, $q_n \equiv 0$.
    \end{itemize}
    
   Similarly, equation \eqref{eq:Conditions_p2n_Weven} means that,
    \begin{itemize}
        \item when $n=2m+1$, $p_{n,2}$ has degree $m$. Thus, $p_{n,2}$ is quasi-orthogonal with respect to $w_1^{(a)}$;
        \item when $n = 2m$,  $p_{n,2}$ has degree $\leq m-1$. If it is exactly $m-1$, then $p_{n,2}$ is orthogonal with respect to $w_1^{(a)}$, otherwise it must be identically equal to $0$.
    \end{itemize}
In summary,
\begin{cor}\label{teo:EvenExpressionsFor_qn_p2n}
        Let $W$ be an even weight on $\gamma_2$. For each $n\in \Z_{\ge 0}$, $m=\floor*{\frac{n}{2}}$ and $k=n-2m\in \{0,1\}$, there exists constants $\kappa_n, \nu_n \in \C$ such that
        \[
        q_n(x) = - \delta_{0,k} P_{m+1}(x,w_{-1}^+) + \kappa_n P_{m+k}(x,w_{-1}^+), \quad 
       p_{n,2}(x) = \delta_{1,k} P_m(x,w_1^+) + \nu_n P_{m-1}(x,w_1^+),
        \]
        where $q_n$ and $p_{n,2}$ are given by Lemma \ref{teo:General_Decomposition_f} and equation \eqref{eq:defqn}, and $\delta_{i,k}$ is the Kronecker delta. 
    \end{cor}

    \begin{remark}
        Further assuming that $a = \omega_1$, one can show that $\kappa_n = \nu_n = 0$, and the decomposition from Lemma \ref{lemma:Decomposition} for $F_n$ reduces to the one given by Theorem \ref{teo:lifting_symmetric}.
    \end{remark}

    \subsection{General decomposition of EOPs} \label{section:generaldecomp} \hfill 

 Now we generalize the construction of Section ~\ref{sec:TotalSymmetry}. 
    Our goal here is to lift the family of orthogonal polynomial with respect to a generic weight $w$ defined on an interval $[e_3,e_2]$ to a family of $a$-EOP with respect to an even weight $W$ on $\gamma_2$ for a suitable torus $\TT$, for a general $a \in \gamma_1$. 

    We assume that $e_2, e_3$ satisfy the conditions \eqref{condtone} and define $e_1=-e_2-e_3$, so that 
    \[
        e_3<e_2<e_1.
    \]
    As in Section \ref{sec:TotalSymmetry}, the triplet $e_3,e_2,e_1$ determines uniquely the Weierstrass elliptic function $\wp = \wp(\cdot, \Lambda)$ for a rectangular lattice $\Lambda$, with corresponding half-periods $\omega_1 > 0$ and $\omega_3 = i \tau$ ,$\tau > 0$.

Given a weight function $w(x)$ supported on $[e_3,e_2]$ with finite moments and $a \in \gamma_1$, consider the weight function
    \[
        W(z) \defeq 
            \left(\wp(a) - \wp(z) \right) w \left(\wp(z) \right) \sqrt{\prod_{l=1}^3 \left(\wp(z) - e_l \right)}, \quad z \in \gamma_2.
    \]
 With the notation \eqref{defweights}, 
    \begin{equation}\label{w_hat}
        w_{-1}^+(x) = w(x), \quad w_{1}^+(x) = \prod_{j=1}^3 (x-e_i) w(x) \eqdef \widehat w(x).
    \end{equation}
    Set $n = 2m$. By Lemma \ref{lemma:Decomposition}, the $n$-th $a$-EOP with respect to $W$ is 
    \[
        F_n(z) = p_{n,1}(\wp(z)) + b_1(z) p_{n,2}(\wp(z)) + \frac{\wp'(a)}{2} p_{3,n}(\wp(z)),
    \]
    where $p_{n,1}$ has degree $m$ and $p_{n,2}$ has degree $\leq m-1$ and, for $q_n$ as in \eqref{eq:defqn},
    \[
        q_n(x) = - P_{m+1}(x,w) + \kappa_n P_m(x, w), \quad p_{n,2}(x) = \nu_n P_{m-1}(x, \widehat w), \quad \kappa_n, \nu_n \in \C.
    \]
    Equation \eqref{eq:defqn} implies that
    \[
        q_n(\wp(a)) = \frac{\wp'(a)}{2} p_{n,2}(\wp(a)) = - P_{m+1}(\wp(a),w) + \kappa_n P_m(\wp(a),w),
    \]
    which gives the following expression for $\kappa_n$,
    \[
        \kappa_n = \frac{\wp'(a)p_{n,2}(\wp(a)) +2P_{m+1}(\wp(a),w)}{2 P_m(\wp(a),w)}.
    \]
    Similarly, we find $\nu_n$ by evaluating $p_{2,n}$ at $\wp(a)$:
    \[
        \nu_n = \frac{p_{n,2}(\wp(a))}{P_{m-1}(\wp(a),\widehat w)}.
    \]

    Equation \eqref{eq:t1_mop_conditions_3} becomes
    \begin{multline*}
        \frac{\wp'(a)}{2} \left( \left[ \frac{\wp'(a)p_{n,2}(\wp(a)) +2P_{m+1}(\wp(a),w)}{2 P_m(\wp(a),w)} \right]\int_{e_3}^{e_2}    \frac{P_m(x,w)w(x)}{\wp(a) - x} \, \dd x 
        \right.\\
        \left. - \int_{e_3}^{e_2}  \frac{P_{m+1}(x,w)w(x)}{\wp(a) - x} \, \dd x \right)+ 
        \frac{p_{n,2}(\wp(a))}{P_{m-1}(\wp(a),\widehat w)} \int_{e_3}^{e_2} \frac{ P_{m-1}(x, \widehat w)\widehat w(x)}{\wp(a) - x} \, \dd x = 0.
    \end{multline*}
    Solving for $p_{n,2}(\wp(a))$ in the expression above we get
    \[
        p_{n,2}(\wp(a)) 
        = \frac{\wp'(a)}{2}
        \left[ T_{m-1}(a) -  \frac{P_{m+1}(\wp(a),w)T_m(a)}{P_m(\wp(a),w)}  \right]
        \left[ \frac{\wp'(a)^2T_m(a)}{4 P_m(\wp(a),w)} 
        + \frac{\widehat T_{m-1}(a)}{P_{m-1}(\wp(a), \widehat w)}
        \right]^{-1} ,
    \]
    where   
    \begin{equation}\label{eq_Rn}
        T_n(a) \defeq \int_{e_3}^{e_2} \frac{P_n(x,w) w(x)}{\wp(a) - x} \, \dd x, \quad 
        \widehat T_n(a) \defeq \int_{e_3}^{e_2} \frac{P_n(x,\widehat w) \widehat w(x)}{\wp(a) - x} \, \dd x.
    \end{equation}
    
   This formula, although cumbersome, is an explicit expression for $p_{n,2}(\wp(a))$ given only in terms of the sequences of monic OPRL with respect to $w$ and $\widehat w$, and their Cauchy transforms evaluated at $\wp(a)$.
    
    The case $n = 2m+1$ is similar: 
    \[
        q_n(x) = \kappa_n P_{m+1} (x,w), \quad p_{n,2}
        (x) = P_m(x, \widehat w) + \nu_n P_{m-1} (x, \widehat w),
    \]
    with
    \[
        \kappa_n = \frac{\wp'(a) p_{n,2}(\wp(a))}{4 P_{m+1} (\wp(a), w)}, \quad \nu_n = \frac{p_{n,2}(\wp(a)) - P_m(\wp(a), \widehat w)}{P_{m-1} (\wp(a), \widehat w)},
    \]
    where
    \[
        p_{n,2}(\wp(a)) = \left[ \frac{P_m(\wp(a),\widehat w)}{P_{m-1}(\wp(a), \widehat w)} \widehat T_{m-1}(a) - \widehat T_m(a) \right] \left[ \frac{\wp'(a)^2}{4} \frac{T_{m+1}(a)}{ P_{m+1}(\wp(a), w)} + \frac{\widehat T_{m-1}(a)}{P_{m-1}(\wp(a), \widehat w)} \right]^{-1}.
    \]
    
    The following theorem summarizes the results of this subsection; it is convenient to remind the reader that $P_n(\cdot, w)$ denotes the $n$-th monic orthogonal polynomial with respect to the weight $w(x) \, \dd x$. 
    \begin{thm}
        Let $\TT$ be a torus with Weierstrass invariants $e_3 < e_2 < e_1$, $\wp(z)$ its corresponding Weierstrass elliptic function, $w$ a positive weight on $[e_3,e_2]$ with finite moments.
             The function
        \[
            W(z) = (\wp(a) - \wp(z)) w(\wp(z)) \sqrt{\prod_{l=1}^3 (\wp(z) - e_l)}, \quad z \in \gamma_2,
        \]
        defines a positive weight on $\gamma_2$.
        
        Let $n\in \Z_{\ge 0}$, $m=\floor*{\frac{n}{2}}$, and $k=n-2m\in \{0,1\}$. Then  the corresponding  $n$-th monic orthogonal $a$-polynomial is  
        \[
            F_n(z) = p_{n,1}(\wp(z)) + b_1(z) p_{n,2}(\wp(z)) + \frac{\wp'(a)}{2}p_{n,3}(\wp(z)),
        \]
 where 
        \begin{align*}
             p_{n,1}(x) &\defeq \frac{ - \delta_{0,k} P_{m+1}(x, w) + \kappa_n P_{m+k}(x, w) - \frac{\wp'(a)}{2} c_n}{  \wp(a) - x}, \\
             p_{n,2}(x) &\defeq \delta_{1,k} P_m(x, \widehat w) + \nu_n P_{m-1}( x, \widehat w), \\
             p_{n,3}(x) &\defeq \frac{p_{n,2}(x) - p_{n,2}(\wp(a))}{x - \wp(a)},  
        \end{align*}
$\widehat w$ is given by \eqref{w_hat}, and  $T_m(a)$ and $\widehat T_m(a)$ as in \eqref{eq_Rn}.

The constants above are
        \[
            c_n = 
            \begin{dcases}
               \frac{\wp'(a)}{2}
        \left[ T_{m-1}(a) -  \frac{P_{m+1}(\wp(a),w)T_m(a)}{P_m(\wp(a),w)}  \right]\left[ \frac{\wp'(a)^2T_m(a)}{4 P_m(\wp(a),w)} 
        + \frac{\widehat T_{m-1}(a)}{P_{m-1}(\wp(a), \widehat w)}
        \right]^{-1}  , &k = 0, \\
               \left[ \frac{P_m(\wp(a),\widehat w)}{P_{m-1}(\wp(a), \widehat w)} \widehat T_{m-1}(a) - \widehat T_m(a) \right] \left[ \frac{\wp'(a)^2}{4} \frac{T_{m+1}(a)}{ P_{m+1}(\wp(a), w)} + \frac{\widehat T_{m-1}(a)}{P_{m-1}(\wp(a), \widehat w)} \right]^{-1} , &k = 1;
            \end{dcases}
        \]
        \[
        \kappa_n = \frac{\wp'(a) c_n + 2 \delta_{k,1} P_{m+1}(\wp(a), w)}{P_{m+k}(\wp(a),w)}; \quad  
        \nu_n = \frac{c_n - \delta_{0,k} P_m (\wp(a), \widehat w)} {P_{m-1}(\wp(a),\widehat w)}.
        \]
    \end{thm}

\section*{Acknowledgments}

    The first author declares that this study was financed, in part, by the São Paulo Research Foundation (FAPESP), Brasil; Process Numbers \#2020/13183-0 and \#2022/12756-1. This work started during his stay as a Visiting Scholar at the Department of Mathematics of Baylor University, Texas, USA. He gratefully acknowledges the department’s generous hospitality.
    
    The second author was partially supported by the Simons Foundation Collaboration Grants for Mathematicians (grant MPS-TSM-00710499)).
    He also acknowledges the support of the project PID2021-124472NB-I00, funded by MICIU/AEI/10.13039/501100011033 and by ``ERDF A way of making Europe'', as well as the support of Junta de Andaluc\'{\i}a (research group FQM-229 and Instituto Interuniversitario Carlos I de F\'{\i}sica Te\'orica y Computacional). 

    Both authors thank Prof. Maxim L. Yattselev and Prof. Guilherme Silva for valuable discussions.



\begin{thebibliography}{10}

\bibitem{Bertola2021a}
M.~Bertola.
\newblock Pad\'{e} approximants on {R}iemann surfaces and {KP} tau functions.
\newblock {\em Anal. Math. Phys.}, 11(4):Paper No. 149, 38, 2021.

\bibitem{Bertola2022a}
M.~Bertola.
\newblock Nonlinear steepest descent approach to orthogonality on elliptic
  curves.
\newblock {\em J. Approx. Theory}, 276:Paper No. 105717, 33, 2022.

\bibitem{DesirajuLahiry2025}
H.~Desiraju and S.~Lahiry.
\newblock Recurrence relations and the {C}hristoffel-{D}arboux formula for
  elliptic orthogonal polynomials.
\newblock Preprint {arXiv:}2506.09582, 2025.

\bibitem{Desiraju2024}
H.~Desiraju, T.~L. Latimer, and P.~Roffelsen.
\newblock On a class of elliptic orthogonal polynomials and their
  integrability.
\newblock {\em Constructive Approximation}, 2024.
\newblock 10.1007/s00365-024-09687-z.

\bibitem{DimitrovIsmailRafaeli2013}
D.~K. Dimitrov, M.~E.~H. Ismail, and F.~R. Rafaeli.
\newblock Interlacing of zeros of orthogonal polynomials under modification of
  the measure.
\newblock {\em J. Approx. Theory}, 175:64--76, 2013.

\bibitem{NIST:DLMF}
DLMF: {\it NIST Digital Library of Mathematical Functions}.
\newblock \url{https://dlmf.nist.gov/}, Release 1.2.4 of 2025-03-15.
\newblock F.~W.~J. Olver, A.~B. {Olde Daalhuis}, D.~W. Lozier, B.~I. Schneider,
  R.~F. Boisvert, C.~W. Clark, B.~R. Miller, B.~V. Saunders, H.~S. Cohl, and
  M.~A. McClain, eds.

\bibitem{DriverJordaanMbuyi2008}
K.~Driver, K.~Jordaan, and N.~Mbuyi.
\newblock Interlacing of the zeros of {J}acobi polynomials with different
  parameters.
\newblock {\em Numer. Algorithms}, 49(1-4):143--152, 2008.

\bibitem{FasondiniOlverXi2023a}
M.~Fasondini, S.~Olver, and Y.~Xu.
\newblock Orthogonal polynomials on a class of planar algebraic curves.
\newblock {\em Stud. Appl. Math.}, 151(1):369--405, 2023.

\bibitem{FasondiniOlverXu23b}
M.~Fasondini, S.~Olver, and Y.~Xu.
\newblock Orthogonal polynomials on planar cubic curves.
\newblock {\em Found. Comput. Math.}, 23(1):1--31, 2023.

\bibitem{Ismail2005}
M.~E.~H. Ismail.
\newblock {\em Classical and quantum orthogonal polynomials in one variable},
  volume~98 of {\em Encyclopedia of Mathematics and its Applications}.
\newblock Cambridge University Press, Cambridge, 2009.

\bibitem{Dunham37}
D.~Jackson.
\newblock Orthogonal polynomials on a plane curve.
\newblock {\em Duke Math. J.}, 3(2):228--236, 1937.

\bibitem{Markushevich}
A.~I. Markushevich.
\newblock {\em Theory of functions of a complex variable. {V}ol. {III}}.
\newblock Prentice-Hall, Inc., Englewood Cliffs, NJ, english edition, 1967.

\bibitem{Miranda}
R.~Miranda.
\newblock {\em Algebraic curves and {R}iemann surfaces}, volume~5 of {\em
  Graduate Studies in Mathematics}.
\newblock American Mathematical Society, Providence, RI, 1995.

\bibitem{OlverXu2019}
S.~Olver and Y.~Xu.
\newblock Orthogonal structure on a wedge and on the boundary of a square.
\newblock {\em Found. Comput. Math.}, 19(3):561--589, 2019.

\bibitem{OlverXu2020}
S.~Olver and Y.~Xu.
\newblock Orthogonal polynomials in and on a quadratic surface of revolution.
\newblock {\em Math. Comp.}, 89(326):2847--2865, 2020.

\bibitem{OlverXu2021}
S.~Olver and Y.~Xu.
\newblock Orthogonal structure on a quadratic curve.
\newblock {\em IMA J. Numer. Anal.}, 41(1):206--246, 2021.

\bibitem{SpicerNijhoff11}
P.~E. Spicer, F.~W. Nijhoff, and P.~H. van~der Kamp.
\newblock Higher analogues of the discrete-time {T}oda equation and the
  quotient-difference algorithm.
\newblock {\em Nonlinearity}, 24(8):2229--2263, 2011.

\bibitem{Szego}
G.~Szeg\H{o}.
\newblock {\em Orthogonal {P}olynomials}, volume Vol. 23 of {\em American
  Mathematical Society Colloquium Publications}.
\newblock American Mathematical Society, New York, 1939.

\end{thebibliography}
\end{document}